\documentclass[12pt,abstracton]{scrartcl}
\usepackage[ansinew]{inputenc}
\usepackage[T1]{fontenc}
\usepackage[english]{babel}
\usepackage{amsmath}
\usepackage{amssymb}
\usepackage{amsthm}
\usepackage{mathrsfs}
\usepackage{array}
\usepackage{subfig}
\usepackage{graphicx}
\usepackage{booktabs}
\usepackage{longtable}
\usepackage{multirow}
\usepackage{marvosym}
\usepackage[flushmargin,hang]{footmisc}
 
\newtheorem{Theorem}{Theorem}[section]
\newtheorem{Corollary}[Theorem]{Corollary}
\newtheorem{Lemma}[Theorem]{Lemma}
\theoremstyle{definition}

\newtheorem{Remark}[Theorem]{Remark}
\newtheorem{Example}[Theorem]{Example}

\numberwithin{equation}{section}
\numberwithin{figure}{section}
\numberwithin{table}{section}

\DeclareMathOperator{\tr}{tr}

\DeclareMathOperator{\eff}{eff}
\DeclareMathOperator{\Diag}{Diag}

\DeclareMathOperator{\Po}{Po}
\DeclareMathOperator{\rank}{rank}

\allowdisplaybreaks 
\addtokomafont{disposition}{\boldmath}

\makeatletter
\g@addto@macro{\thm@space@setup}{\thm@headpunct{}}
\renewenvironment{proof}[1][\proofname]{\par
  \pushQED{\qed}%
  \normalfont \topsep6\p@\@plus6\p@\relax
  \trivlist
  \item[\hskip\labelsep
        \bfseries
    #1\@addpunct{:}]\ignorespaces
}{%
  \popQED\endtrivlist\@endpefalse
}
\makeatother

\makeatletter                                        
\def\blfootnote{\xdef\@thefnmark{}\@footnotetext}
\makeatother

\begin{document}

\title{\vspace{-4pt}Optimal Designs for Poisson Count Data with Gamma Block Effects}
\author{Marius Schmidt \and Rainer Schwabe}
\date{\today}
\maketitle

\vspace{-30pt}
\begin{abstract}
The Poisson-Gamma model is a generalization of the Poisson model, which can be used for modelling count data. We show that the $D$-optimality criterion for the Poisson-Gamma model is equivalent to a combined weighted optimality criterion of $D$-optimality and $D_s$-optimality for the Poisson model. Moreover, we determine the $D$-optimal designs for the Poisson-Gamma model for multiple regression with an arbitrary number of covariates, obtaining the $D_s$-optimal designs for the Poisson and Poisson-Gamma model as a special case. For linear optimality criteria like $L$- and $c$-optimality it is shown that the optimal designs in the Poisson and Poisson-Gamma model coincide.
\end{abstract}

\blfootnote{M. Schmidt (\Letter) $\cdot$ R. Schwabe \\ \noindent Institut für Mathematische Stochastik \\ Otto-von-Guericke-Universität Magdeburg \\ PF 4120, 39016 Magdeburg, Germany \\ E-Mail: marius.schmidt@ovgu.de \\ E-Mail: rainer.schwabe@ovgu.de}

\section{Introduction}

Count data arises in experiments, where the number of objects or occurrences of events of interest is observed. Frequently, the Poisson model is used to model such data, in which the expected value of the Poisson distributed response variable is linked to a linear predictor consisting of covariates and unknown model parameters. In such experiments there may be repeated measurements for each statistical unit. Assuming a Gamma distributed random effect for each statistical unit, we obtain the Poisson-Gamma model as a generalization of the Poisson model.\\
The estimates of the unknown model parameters depend on the choice of the covariates. In order to obtain the most accurate parameter estimates, we determine optimal designs, which specify the optimal values and frequencies of the covariates. With such designs the number of experimental units can be reduced, leading to a lowering of experimental costs. Furthermore, for example in animal testing, the use of optimal designs may be required because of ethical reasons.\\
For the Poisson model Ford et al.\ (1992) and Rodríguez-Torreblanca and Rodríguez-Díaz (2007) determined $D$- and $c$-optimal designs for the case of one covariate. Wang et al.\ (2006) made numerical investigations for two covariates with and without an additional interaction term. For the case of multiple regression with an arbitrary number of covariates Russell et al.\ (2009) derived $D$-optimal designs and Schmidt (2018) determined $c$-, $L$- and $\phi_p$-optimal designs. In the context of intelligence testing Graßhoff et al.\ (2016, 2018) considered the Poisson-Gamma model with one measurement per statistical unit and computed $D$-optimal designs for a binary design region.\\
In Section \ref{Abschnitt Modell} we introduce the Poisson-Gamma model and derive the Fisher information matrix. Section \ref{Abschnitt Optimales Design} gives a brief introduction to the theory of optimal design of experiments and deals with information matrix relations between the Poisson and Poisson-Gamma model. We will be concerned with the determination of $D$-optimal designs for multiple regression with an arbitrary number of covariates in Section \ref{Abschnitt D-optimale Designs} and with optimal designs for linear optimality criteria in Section \ref{Abschnitt Optimale Designs für lineare Optimalitätskriterien}. Since the model under consideration is nonlinear, the optimal designs depend on the unknown parameters and are therefore called locally optimal (cf.\ Chernoff, 1953). We note that most proofs are deferred to an appendix.

\section{The Poisson-Gamma model} \label{Abschnitt Modell}

We consider $n$ statistical units, for example groups or individuals, for each of which $m$ experiments with response variables $Y_{ij}$, $i=1,\ldots,n$, $j=1,\ldots,m$, are performed. To each statistical unit a Gamma distributed block effect $\Theta_i\sim\gamma(a,b)$ with known shape parameter $a>0$ and known rate parameter $b>0$ is assigned. The probability density function of the Gamma distribution $\gamma(a,b)$ is given by $f_{\gamma}(\theta)=b^a\cdot\Gamma(a)^{-1}\cdot\theta^{a-1}\cdot e^{-b\cdot\theta}$ for $\theta>0$, where $\Gamma(a)$ denotes the Gamma function, which satisfies $\Gamma(a+1)=a\cdot\Gamma(a)$. We assume that given $\Theta_i=\theta_i$ the random variables $Y_{ij}$ are independent Poisson distributed with parameter $\lambda_{ij}$ depending on $\theta_i$. The expected value $\lambda_{ij}$ is related via the canonical link function to the linear predictor, which consists of a fixed effects term $\boldsymbol{f}(\boldsymbol{x}_{ij})^T\boldsymbol{\beta}$ and an additive random effect $v_i=\ln(\theta_i)$:
\begin{align}
\ln(\lambda_{ij})=\boldsymbol{f}(\boldsymbol{x}_{ij})^T\boldsymbol{\beta}+v_i.
\end{align} 
It follows that $\lambda_{ij}=\theta_i\cdot\exp(\boldsymbol{f}(\boldsymbol{x}_{ij})^T\boldsymbol{\beta})$. Here $\boldsymbol{x}_{ij}\in\mathbb{R}^k$ is the vector of covariates, the vector $\boldsymbol{f}=(1,f_1,\ldots,f_{p-1})^T$ consists of known regression functions and the vector $\boldsymbol{\beta}=(\beta_0,\ldots,\beta_{p-1})^T$ is the unknown parameter vector.\\
In the following, an arbitrary statistical unit $i$ is considered. For simplicity of notation, the index $i$ is suppressed. The Poisson and Gamma distribution are conjugate distributions and the probability density function of $\boldsymbol{Y}=(Y_1,\ldots,Y_m)$ can be derived analytically.
\begin{Theorem}\label{Satz Dichte Poisson-Gamma-Modell}
The probability density function of $\boldsymbol{Y}=(Y_1,\ldots,Y_m)$ is given by
\begin{align}
f_{\boldsymbol{Y}}(\boldsymbol{y})=\frac{\Gamma\bigl(a+\sum_{j=1}^m y_j\bigr)}{\Gamma(a)\cdot\prod_{j=1}^m y_j!}\cdot\frac{e^{\sum_{j=1}^m \boldsymbol{f}(\boldsymbol{x}_j)^T\boldsymbol{\beta}\cdot y_j}}{\left(b+\sum_{j=1}^m e^{\boldsymbol{f}(\boldsymbol{x}_j)^T\boldsymbol{\beta}}\right)^{\sum_{j=1}^m y_j}}\cdot\left(\frac{b}{b+\sum_{j=1}^m e^{\boldsymbol{f}(\boldsymbol{x}_j)^T\boldsymbol{\beta}}}\right)^a.
\end{align}
\end{Theorem} 
Since the marginal random variable $Y_j$, $j=1,\ldots,m$, is Poisson-Gamma distributed, the expectation of $Y_j$ is given by $E(Y_j)=\frac{a}{b}\cdot\exp(\boldsymbol{f}(\boldsymbol{x}_j)^T\boldsymbol{\beta})$. Using Theorem \ref{Satz Dichte Poisson-Gamma-Modell} we derive the Fisher information matrix for a single statistical unit in the next theorem, where $\boldsymbol{e}_1\in\mathbb{R}^p$ denotes the first standard unit vector.
\begin{Theorem}\label{Satz Fisher-Informationsmatrix Poisson-Gamma-Modell}
The Fisher information matrix for the parameter vector $\boldsymbol{\beta}$ is given by
\begin{align}
\boldsymbol{I}(\boldsymbol{\beta})=\frac{a}{b}\cdot\left(\boldsymbol{I}_{\Po}(\boldsymbol{\beta})-\frac{\boldsymbol{I}_{\Po}(\boldsymbol{\beta})\boldsymbol{e}_1\boldsymbol{e}_1^T\boldsymbol{I}_{\Po}(\boldsymbol{\beta})}{\boldsymbol{e}_1^T\boldsymbol{I}_{\Po}(\boldsymbol{\beta})\boldsymbol{e}_1+b}\right),\label{Fisher-Informationsmatrix}
\end{align}
where $\boldsymbol{I}_{\Po}(\boldsymbol{\beta})=\sum_{j=1}^m \exp\bigl(\boldsymbol{f}(\boldsymbol{x}_j)^T\boldsymbol{\beta}\bigr)\cdot\boldsymbol{f}(x_j)\boldsymbol{f}(x_j)^T$ is the Fisher information matrix for the Poisson model.
\end{Theorem}
Since the observations are independent between the statistical units, the Fisher information matrix $\boldsymbol{I}_{\text{Total}}(\boldsymbol{\beta})$ for $n$ statistical units is the sum of the Fisher information matrices $\boldsymbol{I}_i(\boldsymbol{\beta})$ for each statistical unit $i$, that is $\boldsymbol{I}_{\text{Total}}(\boldsymbol{\beta})=\sum_{i=1}^n \boldsymbol{I}_i(\boldsymbol{\beta})$.\\
If only one observation per statistical unit is considered, that is for $m=1$, the generalized negative binomial model results (cf.\ Graßhoff et al., 2016), for which the Fisher information matrix is given by:
\begin{align}
\boldsymbol{I}_{\text{Total}}(\boldsymbol{\beta})=\sum_{i=1}^n \frac{a\cdot e^{\boldsymbol{f}(\boldsymbol{x}_i)^T\boldsymbol{\beta}}}{e^{\boldsymbol{f}(\boldsymbol{x}_i)^T\boldsymbol{\beta}}+b}\cdot \boldsymbol{f}(\boldsymbol{x}_i)\boldsymbol{f}(\boldsymbol{x}_i)^T.
\end{align}
Due to the random block effect $v_i=\ln(\theta_i)$, the random variables $Y_1,\ldots,Y_m$ for a statistical unit are not independent. Therefore, the Fisher information matrix \eqref{Fisher-Informationsmatrix} for a statistical unit cannot be represented as the sum of the Fisher information matrices for each observation.

\section{Design, information and optimality criteria} \label{Abschnitt Optimales Design}

The quality of the parameter estimates depends on the choice of experimental settings. In order to estimate the parameters as precisely as possible, the experimental settings have to be chosen optimally in a certain sense. First, we consider a single statistical unit. A design consists of different experimental settings $\boldsymbol{x}_1,\ldots,\boldsymbol{x}_l\in\mathscr{X}$ with replications $r_j\in\mathbb{N}$, $\sum_{j=1}^l r_j=m$, where $\mathscr{X}\subset\mathbb{R}^k$ is the design region. Instead of the replications, relative frequencies $w_j=r_j/m$ are often considered, which indicate how frequently the corresponding experimental setting is used for a statistical unit. This concept is generalized to that of approximate individual designs
\begin{align}
\xi=\begin{Bmatrix}\boldsymbol{x}_1 & \ldots & \boldsymbol{x}_l\\w_1 & \ldots & w_l\end{Bmatrix},
\end{align}
which are probability measures on $\mathscr{X}$ with finite support (cf.\ Silvey, 1980, p.\ 15). Such a design assigns arbitrary weights $0\leq w_1,\ldots,w_l\leq1$ with $\sum_{j=1}^l w_j=1$ to the experimental settings. We denote the set of all approximate designs $\xi$ on $\mathscr{X}$ by $\Xi$. The information matrix $\boldsymbol{M}(\xi;\boldsymbol{\beta})$ for a design $\xi$ is obtained by standardising the Fisher information matrix with the number of observations $m$ and allowing continuous weights.\\
For the entire experiment with $n$ statistical units the population design
\begin{align}
\zeta=\begin{Bmatrix}\xi_1 & \ldots & \xi_r\\q_1 & \ldots & q_r\end{Bmatrix}
\end{align}
consists of the individual designs $\xi_i$ and the corresponding weights $0\leq q_i\leq1$ with $\sum_{i=1}^r q_i=1$. These weights are the proportions of the statistical units obtaining the individual design $\xi_i$. The observations between the statistical units are independent, hence the information matrix for the population design $\zeta$ can be obtained as $\boldsymbol{M}(\zeta;\boldsymbol{\beta})=\sum_{i=1}^r q_i\boldsymbol{M}(\xi_i;\boldsymbol{\beta})$.\\
Optimal designs are based on the optimization of a real-valued function $\Phi$ of the information matrix with respect to the design (cf.\ Silvey, 1980, p.\ 10). We introduce some commonly used optimality criteria with respect to individual designs $\xi$. For population designs $\zeta$ the optimality criteria can be defined analogously.\\
One of the most popular optimality criteria is $D$-optimality. A design $\xi^{\ast}$ with regular information matrix $\boldsymbol{M}(\xi^{\ast};\boldsymbol{\beta})$ is $D$-optimal if $\det\bigl(\boldsymbol{M}(\xi^\ast;\boldsymbol{\beta})\bigr)\geq\det\bigl(\boldsymbol{M}(\xi;\boldsymbol{\beta})\bigr)$ holds for all $\xi\in\Xi$. The $D$-optimal design minimizes the volume of the confidence ellipsoid for the parameters (cf.\ Silvey, 1980, p.\ 10).\\
If not the entire parameter vector is to be estimated, but certain linear combinations $\boldsymbol{A}^T \boldsymbol{\beta}$, where $\boldsymbol{A}$ is a $p\times s\thinspace$-matrix with $\rank(\boldsymbol{A})=s<p$, then the information matrix of the optimal design need not be regular. Therefore, the concept of identifiability is introduced. Given $\boldsymbol{\beta}$ the linear combinations $\boldsymbol{A}^T \boldsymbol{\beta}$ are identifiable for a design $\xi$ if $\boldsymbol{A}=\boldsymbol{M}(\xi;\boldsymbol{\beta}) \hspace{1pt} \boldsymbol{H}$ holds for a matrix $\displaystyle \boldsymbol{H}\in\mathbb{R}^{p\times s}$ (cf.\ Silvey, 1980, p.\ 25).\\
To estimate $\boldsymbol{A}^T \boldsymbol{\beta}$, the identifiability condition has to be satisfied. The $D_A$-optimality criterion can be used to compute optimal designs for estimation of $\boldsymbol{A}^T \boldsymbol{\beta}$. A design $\xi^{\ast}$ is $D_A$-optimal if $\boldsymbol{A}^T \boldsymbol{\beta}$ is identifiable and $\det\bigl(\boldsymbol{A}^T \boldsymbol{M}(\xi^\ast;\boldsymbol{\beta})^{-} \boldsymbol{A}\bigr)\leq\det\bigl(\boldsymbol{A}^T \boldsymbol{M}(\xi;\boldsymbol{\beta})^{-} \boldsymbol{A}\bigr)$ holds for all $\xi\in\Xi$ for which $\boldsymbol{A}^T \boldsymbol{\beta}$ is identifiable. Here $\boldsymbol{M}(\xi;\boldsymbol{\beta})^-$ is a generalized inverse of $\boldsymbol{M}(\xi;\boldsymbol{\beta})$. If only $s$ individual parameters are of interest, then the criterion is called $D_s$-optimality. For example, for $\boldsymbol{A}^T=(\boldsymbol{0}_{p-1}, \boldsymbol{I}_{p-1})$ we have $D_s$-optimality for the $s=p-1$ parameters $\beta_1,\ldots,\beta_{p-1}$ (cf.\ Silvey, 1980, p.\ 11, 26).\\
A linear optimality criterion is $L$-optimality. Let $\boldsymbol{B}=\boldsymbol{A}\boldsymbol{A}^T$ be a symmetric positive definite matrix. A design $\xi^{\ast}$ is $L$-optimal if $\boldsymbol{A}^T \boldsymbol{\beta}$ is identifiable and $\tr\bigl(\boldsymbol{M}(\xi^{\ast};\boldsymbol{\beta})^-\boldsymbol{B}\bigr)\leq\tr\bigl(\boldsymbol{M}(\xi;\boldsymbol{\beta})^-\boldsymbol{B}\bigr)$ holds for all $\xi\in\Xi$ for which $\boldsymbol{A}^T \boldsymbol{\beta}$ is identifiable. If $\boldsymbol{B}=\boldsymbol{I}$ is the identity matrix, then $A$-optimality results. An $A$-optimal design minimizes the sum of the asymptotic variances of the estimators for the individual components of the parameter vector. For $\boldsymbol{B}=\boldsymbol{c}\boldsymbol{c}^T$ with $\boldsymbol{c}\in\mathbb{R}^p$ we obtain $c$-optimality, for which the criterion function can be written as $\boldsymbol{c}^T \boldsymbol{M}(\xi;\boldsymbol{\beta})^- \boldsymbol{c}$. The $c$-optimality criterion aims at estimating the linear combination $\boldsymbol{c}^T \boldsymbol{\beta}$ with minimal asymptotic variance (cf.\ Atkinson et al., 2007, p.\ 142--143).\\
Now we consider the Poisson-Gamma model. From Theorem \ref{Satz Fisher-Informationsmatrix Poisson-Gamma-Modell} we generalize the Fisher information matrix to the information matrix of an individual design $\xi$:
\begin{align}
\boldsymbol{M}(\xi;\boldsymbol{\beta})=\frac{a}{b}\cdot\left(\boldsymbol{M}_{\Po}(\xi;\boldsymbol{\beta})-\frac{\boldsymbol{M}_{\Po}(\xi;\boldsymbol{\beta})\boldsymbol{e}_1\boldsymbol{e}_1^T\boldsymbol{M}_{\Po}(\xi;\boldsymbol{\beta})}{\boldsymbol{e}_1^T\boldsymbol{M}_{\Po}(\xi;\boldsymbol{\beta})\boldsymbol{e}_1+\frac{b}{m}}\right).\label{Darstellung Informationsmatrix Poisson-Gamma-Modell}
\end{align}
Here $\boldsymbol{M}_{\Po}(\xi;\boldsymbol{\beta})=\sum_{j=1}^l w_j\cdot\exp\bigl(\boldsymbol{f}(\boldsymbol{x}_j)^T\boldsymbol{\beta}\bigr)\cdot \boldsymbol{f}(\boldsymbol{x}_j)\boldsymbol{f}(\boldsymbol{x}_j)^T$ is the information matrix for the Poisson model. With the design matrix $\boldsymbol{X}=\big(\boldsymbol{f}(\boldsymbol{x}_1),\ldots,\boldsymbol{f}(\boldsymbol{x}_l)\big)^T$ and the diagonal matrices $\boldsymbol{W}=\Diag(w_1,\ldots,w_l)$ and $\boldsymbol{\Lambda}=\Diag\bigl(\exp\bigl(\boldsymbol{f}(\boldsymbol{x}_1)^T\boldsymbol{\beta}\bigr),\ldots,\exp\bigl(\boldsymbol{f}(\boldsymbol{x}_l)^T\boldsymbol{\beta}\bigr)\bigr)$ the information matrix of the Poisson model can be written as $\boldsymbol{M}_{\Po}(\xi;\boldsymbol{\beta})=\boldsymbol{X}^T\boldsymbol{W}\boldsymbol{\Lambda}\boldsymbol{X}$.\\
Since $a>0$ is a multiplicative factor in the information matrix for the Poisson-Gamma model, an optimal design does not depend on $a$. Based on Lemma \ref{Lemma Matrix für Darstellung Informationsmatrix Poisson-Gamma-Modell} in the Appendix we obtain the following relations between the information matrices of the Poisson and Poisson-Gamma model.
\begin{Lemma}\label{Lemma Rang Zusammenhang Informationsmatrix}
For a design $\xi$ the information matrices $\boldsymbol{M}(\xi;\boldsymbol{\beta})$ in the Poisson-Gamma model and $\boldsymbol{M}_{\Po}(\xi;\boldsymbol{\beta})$ in the Poisson model have the same rank.
\end{Lemma}
\begin{Theorem}\label{Satz Identifizierbarkeit}
Let $\boldsymbol{A}$ be a $p\times s$-matrix with $\rank(\boldsymbol{A})=s\leq p$. For a design $\xi$ the linear combinations $\boldsymbol{A}^T \boldsymbol{\beta}$ are identifiable in the Poisson-Gamma model if and only if the linear combinations $\boldsymbol{A}^T \boldsymbol{\beta}$ are identifiable for the design $\xi$ in the Poisson model.
\end{Theorem}
\begin{Lemma}\label{Lemma Verallgemeinerte Inverse Informationsmatrix Poisson-Gamma-Modell}
The matrix
\begin{align}
\boldsymbol{M}(\xi;\boldsymbol{\beta})^-=\frac{b}{a}\cdot \boldsymbol{M}_{\Po}(\xi;\boldsymbol{\beta})^-+\frac{m}{a}\cdot \boldsymbol{e}_1\boldsymbol{e}_1^T
\end{align}
is a generalized inverse of $\boldsymbol{M}(\xi;\boldsymbol{\beta})$ if and only if $\boldsymbol{M}_{\Po}(\xi;\boldsymbol{\beta})^-$ is a generalized inverse of $\boldsymbol{M}_{\Po}(\xi;\boldsymbol{\beta})$.
\end{Lemma}
\begin{Remark}\label{Bemerkung Inverse Informationsmatrix Poisson-Gamma-Modell}
By Lemma \ref{Lemma Rang Zusammenhang Informationsmatrix} $\boldsymbol{M}(\xi;\boldsymbol{\beta})$ is regular if and only if $\boldsymbol{M}_{\Po}(\xi;\boldsymbol{\beta})$ is regular. In this case, the following relation for the inverses of the information matrices follows from Lemma \ref{Lemma Verallgemeinerte Inverse Informationsmatrix Poisson-Gamma-Modell}:
\begin{align}
\boldsymbol{M}(\xi;\boldsymbol{\beta})^{-1}=\frac{b}{a}\cdot \boldsymbol{M}_{\Po}(\xi;\boldsymbol{\beta})^{-1}+\frac{m}{a}\cdot \boldsymbol{e}_1\boldsymbol{e}_1^T.
\end{align}
\end{Remark}
For generalized linear models like the Poisson model the information matrix of a convex combination of designs is equal to the convex combination of the information matrices of these designs (cf.\ Fedorov, 1972, p.\ 66). Due to the random effect this does not hold for the Poisson-Gamma model. Since the information matrix can be represented as $\boldsymbol{M}(\xi;\boldsymbol{\beta})=\big(\boldsymbol{\tilde{M}}(\xi;\boldsymbol{\beta})^{-1}+\boldsymbol{D}\big)^{-1}$ with $\boldsymbol{\tilde{M}}(\xi;\boldsymbol{\beta})=\frac{a}{b}\cdot\boldsymbol{M}_{\Po}(\xi;\boldsymbol{\beta})$ and $\boldsymbol{D}=\frac{m}{a}\cdot \boldsymbol{e}_1\boldsymbol{e}_1^T$ not depending on the design $\xi$, the following result can be shown (cf.\ Schmelter (2007), Niaparast (2009)).
\begin{Theorem}\label{Satz Ungleichung Informationsmatrix}
Let $\xi_1$ and $\xi_2$ be two designs. Then the following inequality with respect to the Loewner order holds for all $\alpha\in\left[0,1\right]$:
\begin{align}
\boldsymbol{M}\bigl(\alpha\cdot\xi_1+(1-\alpha)\cdot\xi_2;\boldsymbol{\beta}\bigr)\geq\alpha\cdot\boldsymbol{M}(\xi_1;\boldsymbol{\beta})+(1-\alpha)\cdot\boldsymbol{M}(\xi_2;\boldsymbol{\beta}).
\end{align}
\end{Theorem}
A criterion function $\Phi$ is called isotonic if $\Phi\left(\boldsymbol{M}_1\right)\geq\Phi\left(\boldsymbol{M}_2\right)$ holds for all positive semidefinite matrices $\boldsymbol{M}_1\geq\boldsymbol{M}_2$. We note that each optimality criterion under consideration can be transformed into a maximization problem with an isotonic criterion function (cf.\ Pronzato and Pázman, 2013, p.\ 114, 118).
\begin{Corollary}\label{Korollar Konstruktion besseres Design}
Let $\zeta=\genfrac{\{}{.}{0pt}{}{\xi_1}{q_1}\genfrac{}{}{0pt}{}{\ldots}{\ldots}\genfrac{.}{\}}{0pt}{}{\xi_r}{q_r}$ be an arbitrary population design. Then $\boldsymbol{M}(\tilde{\zeta};\boldsymbol{\beta})\geq\boldsymbol{M}(\zeta;\boldsymbol{\beta})$ holds for the population design $\tilde{\zeta}=\genfrac{\{}{\}}{0pt}{}{\xi}{1}$, which assigns weight 1 to the individual design $\xi=\sum_{i=1}^r q_i\xi_i$. Hence, for an isotonic optimality criterion $\Phi\bigl(\boldsymbol{M}(\tilde{\zeta};\boldsymbol{\beta})\bigr)\geq\Phi\bigl(\boldsymbol{M}(\zeta;\boldsymbol{\beta})\bigr)$ holds.
\end{Corollary}
Since by Corollary \ref{Korollar Konstruktion besseres Design} an optimal individual design $\xi^{\ast}$ yields an optimal population design $\tilde{\zeta}^{\ast}$, which uses $\xi^{\ast}$ for all statistical units (cf.\ Schmelter, 2007), we can restrict ourselves to the determination of optimal individual designs.\\
An important tool to prove the optimality of a design is the equivalence theorem. With Theorem \ref{Satz Ungleichung Informationsmatrix} it follows that for an isotonic and concave optimality criterion $\Phi$, such as $\Phi(\cdot)=\log(\det(\cdot))$ for $D$-optimality, the function $\Psi(\xi)=\Phi(\boldsymbol{M}(\xi;\boldsymbol{\beta}))$ is also concave on $\Xi$, which is a necessary condition for deriving equivalence theorems. That for $D$-optimality is stated in the following theorem (cf.\ Fedorov and Hackl, 1997, p.\ 78).
\begin{Theorem}[Equivalence Theorem]\label{Satz Äquivalenzsatz}
Let the information matrix be given by $\boldsymbol{M}(\xi;\boldsymbol{\beta})=(\boldsymbol{\tilde{M}}(\xi;\boldsymbol{\beta})^{-1}+\boldsymbol{D})^{-1}$, where $\boldsymbol{\tilde{M}}(\xi;\boldsymbol{\beta})=\sum_{j=1}^l w_j\cdot\lambda\bigl(\boldsymbol{f}(\boldsymbol{x}_j)^T\boldsymbol{\beta}\bigr)\cdot\boldsymbol{f}(\boldsymbol{x}_j)\boldsymbol{f}(\boldsymbol{x}_j)^T$. A design $\xi^{\ast}$ is $D$-optimal if and only if
\begin{align}
\lambda\bigl(\boldsymbol{f}(\boldsymbol{x})^T\boldsymbol{\beta}\bigr)\cdot \boldsymbol{f}(\boldsymbol{x})^T\boldsymbol{\tilde{M}}(\xi^{\ast};\boldsymbol{\beta})^{-1}\boldsymbol{M}(\xi^{\ast};\boldsymbol{\beta})\boldsymbol{\tilde{M}}(\xi^{\ast};\boldsymbol{\beta})^{-1}\boldsymbol{f}(\boldsymbol{x})\leq\tr\bigl(\boldsymbol{M}(\xi^{\ast};\boldsymbol{\beta})\boldsymbol{\tilde{M}}(\xi^{\ast};\boldsymbol{\beta})^{-1}\bigr)\notag
\end{align}
for all $\boldsymbol{x}\in\mathscr{X}$. At the support points of $\xi^{\ast}$ equality holds.
\end{Theorem}
The quality of a design $\xi$ can be measured by its efficiency, which is the ratio of the values of the homogeneous version of the criterion function for $\xi$ and for the optimal design $\xi^{\ast}$. For example, $\eff_D(\xi;\boldsymbol{\beta})=\left[\det\bigl(\boldsymbol{M}(\xi;\boldsymbol{\beta})\bigr)/\det\bigl(\boldsymbol{M}(\xi^{\ast};\boldsymbol{\beta})\bigr)\right]^{1/p}$ is the $D$-efficiency and $\eff_{D_A}(\xi;\boldsymbol{\beta})=\left[\det\bigl(\boldsymbol{A}^T \boldsymbol{M}(\xi^\ast;\boldsymbol{\beta})^{-}\boldsymbol{A}\bigr)/\det\bigl(\boldsymbol{A}^T \boldsymbol{M}(\xi;\boldsymbol{\beta})^{-} \boldsymbol{A}\bigr)\right]^{1/s}$ is the $D_A$-efficiency (cf.\ Atkinson et al., 2007, p.\ 151).

\section{$D$-optimal designs} \label{Abschnitt D-optimale Designs}

First, using Remark \ref{Bemerkung Inverse Informationsmatrix Poisson-Gamma-Modell} we determine the criterion function for $D$-optimality and establish a relation between $D$-optimality for the Poisson and the Poisson-Gamma model.
\begin{Theorem}\label{Satz Kriteriumsfunktion D-Optimalität Poisson-Gamma-Modell}
The $D$-optimality criterion function for the Poisson-Gamma model is given by
\begin{align}
\det\bigl(\boldsymbol{M}(\xi;\boldsymbol{\beta})\bigr)=\left(\frac{a}{b}\right)^p\cdot\frac{\det\bigl(\boldsymbol{M}_{\Po}(\xi;\boldsymbol{\beta})\bigr)}{1+\frac{m}{b}\cdot\boldsymbol{e}_1^T\boldsymbol{M}_{\Po}(\xi;\boldsymbol{\beta})\boldsymbol{e}_1}.\label{Determinante Informationsmatrix Poisson-Gamma-Modell}
\end{align}
\end{Theorem}
The maximization of $\det\bigl(\boldsymbol{M}(\xi;\boldsymbol{\beta})\bigr)$ is equivalent to the minimization of the inverse determinant, which is given by
\begin{align}
\det\bigl(\boldsymbol{M}(\xi;\boldsymbol{\beta})\bigr)^{-1}=\left(\frac{b}{a}\right)^p\cdot\left(\det\bigl(\boldsymbol{M}_{\Po}(\xi;\boldsymbol{\beta})\bigr)^{-1}+\frac{m}{b}\cdot\frac{\boldsymbol{e}_1^T\boldsymbol{M}_{\Po}(\xi;\boldsymbol{\beta})\boldsymbol{e}_1}{\det\bigl(\boldsymbol{M}_{\Po}(\xi;\boldsymbol{\beta})\bigr)}\right).\notag
\end{align}
The criterion function for $D_s$-optimality for the parameters $\beta_1,\ldots,\beta_{p-1}$, that is with $\boldsymbol{A}^T=(\boldsymbol{0}_{p-1},\boldsymbol{I}_{p-1})$, for the Poisson model is given by (cf.\ Atkinson et al., 2007, p.\ 139):
\begin{align}
\det\bigl(\boldsymbol{A}^T \boldsymbol{M}_{\Po}(\xi;\boldsymbol{\beta})^{-1} \boldsymbol{A}\bigr)=\frac{\boldsymbol{e}_1^T\boldsymbol{M}_{\Po}(\xi;\boldsymbol{\beta})\boldsymbol{e}_1}{\det\bigl(\boldsymbol{M}_{\Po}(\xi;\boldsymbol{\beta})\bigr)}.\notag
\end{align}
Thus we obtain the following relation:
\begin{align}
\max_{\xi\in\Xi}\left\{\det\bigl(\boldsymbol{M}(\xi;\boldsymbol{\beta})\bigr)\right\}\;\;\Leftrightarrow\;\;\min_{\xi\in\Xi}\left\{\det\bigl(\boldsymbol{M}_{\Po}(\xi;\boldsymbol{\beta})\bigr)^{-1}+\frac{m}{b}\cdot\det\bigl(\boldsymbol{A}^T\boldsymbol{M}_{\Po}(\xi;\boldsymbol{\beta})^{-1}\boldsymbol{A}\bigr)\right\}.\notag
\end{align}
Hence, the $D$-optimality criterion for the Poisson-Gamma model is equivalent to a combined weighted optimality criterion of $D$-optimality and $D_s$-optimality for $\beta_1,\ldots,\beta_{p-1}$ for the Poisson model.\\
In the following, we consider the multiple regression model with regression function $\boldsymbol{f}(\boldsymbol{x})=(1,\boldsymbol{x}^T)^T$, where $\boldsymbol{x}\in\mathbb{R}^{p-1}$, and parameter vector $\boldsymbol{\beta}=(\beta_0,\beta_1,\ldots,\beta_{p-1})^T$. The next theorem provides the $D$-optimal weights for a design with minimal support.
\begin{Theorem}\label{Satz D-optimale Gewichte Poisson-Gamma-Modell}
Let $\boldsymbol{x}_1,\ldots,\boldsymbol{x}_p$ be linearly independent support points of a design $\xi$, where $\boldsymbol{x}_1,\ldots,\boldsymbol{x}_{p-1}$ are located on a hyperplane $\left\{\boldsymbol{x}\in\mathbb{R}^{p-1}:\boldsymbol{f}(\boldsymbol{x})^T\boldsymbol{\beta}=c\right\}$. For $\boldsymbol{x}_p$ let the inequality $\boldsymbol{f}(\boldsymbol{x}_p)^T\boldsymbol{\beta}>c$ hold. Then the $D$-optimal weight for $\boldsymbol{x}_p$ is given by
\begin{align}
w_p^{\ast}=\frac{2}{p+\sqrt{(p-2)^2+4\cdot(p-1)\cdot\frac{\textstyle1+\frac{m}{b}\cdot\exp(\boldsymbol{f}(\boldsymbol{x}_p)^T\boldsymbol{\beta})}{\textstyle1+\frac{m}{b}\cdot\exp(c)}}}.
\end{align}
The $D$-optimal weights for $\boldsymbol{x}_1,\ldots,\boldsymbol{x}_{p-1}$ are given by $w_1^{\ast}=\ldots=w_{p-1}^{\ast}=(1-w_p^{\ast})/(p-1)$. The $D$-optimal weights satisfy the inequality:
\begin{align}
0<w_p^{\ast}<\frac{1}{p}<w_1^{\ast}=\ldots=w_{p-1}^{\ast}<\frac{1}{p-1}.
\end{align}
\end{Theorem}
\begin{Remark}\label{Bemerkung Design}
We consider the rectangular design region $\mathscr{X}=\left[u_1,v_1\right]\times\ldots\times\left[u_{p-1},v_{p-1}\right]$ and $\beta_i\neq0$ for $i=1,\ldots,p-1$. Let $\boldsymbol{d}=(d_1,\ldots,d_{p-1})$ with $d_i=v_i$ for $\beta_i>0$ and $d_i=u_i$ for $\beta_i<0$. Let $\boldsymbol{e}_i\in\mathbb{R}^{p-1}$ denote the $i$-th standard unit vector and let $z>0$. A design with one support point at the vertex $\boldsymbol{d}$ of $\mathscr{X}$ and further support points $\boldsymbol{d}-(z/\beta_1)\cdot \boldsymbol{e}_1,\ldots,\boldsymbol{d}-(z/\beta_{p-1})\cdot \boldsymbol{e}_{p-1}$ on the edges satisfies the conditions of Theorem~\ref{Satz D-optimale Gewichte Poisson-Gamma-Modell} with $c=\boldsymbol{f}(\boldsymbol{d})^T\boldsymbol{\beta}-z$. Thus the $D$-optimal weights $w_1^{\ast}(z),\ldots,w_p^{\ast}(z)$ from Theorem~\ref{Satz D-optimale Gewichte Poisson-Gamma-Modell} depend on the distance $z$ to the vertex $\boldsymbol{d}$.
\end{Remark}
\begin{Theorem}\label{Satz D-optimale Design Poisson-Gamma-Modell}
Let $\mathscr{X}=\left[u_1,v_1\right]\times\ldots\times\left[u_{p-1},v_{p-1}\right]$ and $\beta_i\neq0$ for $i=1,\ldots,p-1$. Let $\boldsymbol{d}=(d_1,\ldots,d_{p-1})$ with $d_i=v_i$ for $\beta_i>0$ and $d_i=u_i$ for $\beta_i<0$. For any $z>0$ let $w_1^{\ast}(z)$ and $w_p^{\ast}(z)$ be the $D$-optimal weights of Remark \ref{Bemerkung Design}. The equation
\begin{align}
0&=m\cdot\left((p-1)\cdot w_1^{\ast}(z)\cdot e^{\boldsymbol{f}(\boldsymbol{d})^T\boldsymbol{\beta}-z}+w_p^{\ast}(z)\cdot e^{\boldsymbol{f}(\boldsymbol{d})^T\boldsymbol{\beta}}\right)\cdot\big(z\cdot(p-1)\cdot w_1^{\ast}(z)-2\big)\notag\\&\;\;\;\;+b\cdot\left(z\cdot p\cdot w_1^{\ast}(z)-2\right)\label{Gleichung Poisson-Gamma-Modell}
\end{align}
has a unique solution $z^{\ast}$ in the interval $(0,\infty)$. If $z^{\ast}\leq\min_{i=1,\ldots,p-1}\bigl(\left|\beta_i\right|\cdot(v_i-u_i)\bigr)$, then the $D$-optimal design $\xi^{\ast}$ is given by
\begin{align}
\xi^{\ast}=\begin{Bmatrix}\boldsymbol{d}-(z^{\ast}/\beta_1)\cdot \boldsymbol{e}_1 & \ldots & \boldsymbol{d}-(z^{\ast}/\beta_{p-1})\cdot \boldsymbol{e}_{p-1} & \boldsymbol{d}\\[2pt]w_1^{\ast}(z^{\ast}) & \ldots & w_{p-1}^{\ast}(z^{\ast}) & w_p^{\ast}(z^{\ast})\end{Bmatrix}.
\end{align}
\end{Theorem}
The structure of the $D$-optimal design for the Poisson-Gamma model, which is illustrated for the case of three covariates in Figure \ref{Abbildung D-optimales Design}, is similar to that for the Poisson model. One support point is located at the vertex $\boldsymbol{d}$ of the design region, where $\exp\bigl(\boldsymbol{f}(\boldsymbol{x})^T\boldsymbol{\beta}\bigr)$ is maximal. Since $z^{\ast}\leq\min_{i=1,\ldots,p-1}\bigl(\left|\beta_i\right|\cdot(v_i-u_i)\bigr)$ holds, the other support points are located within the design region. They lie on the edges, which are adjacent to $\boldsymbol{d}$. For the Poisson-Gamma model the distance from $\boldsymbol{d}$ to the other support points is given by $z^{\ast}/\beta_i$. Equation \eqref{Gleichung Poisson-Gamma-Modell} can only be satisfied if $z\cdot(p-1)\cdot w_1^{\ast}(z)-2<0$ and $z\cdot p\cdot w_1^{\ast}(z)-2>0$ hold. Since $1/p<w_1^{\ast}(z)<1/(p-1)$ by Theorem \ref{Satz D-optimale Gewichte Poisson-Gamma-Modell}, it follows that $2\cdot(p-1)/p<z^{\ast}<2\cdot p/(p-1)$. The optimal weights for the first $p-1$ support points are equal, but differ from that for the vertex $\boldsymbol{d}$. This is a difference to the $D$-optimal design for the Poisson model, where all weights are equal. Furthermore, for the Poisson model the distance from the vertex $\boldsymbol{d}$ to the other support points is given by $2/\beta_i$ (cf.\ Russell et al., 2009), i.e.\ $z^{\ast}=2$.
\begin{figure}[!htb]
\centering
\includegraphics[width=8.4cm, trim= 46 32 38 22, clip=true]{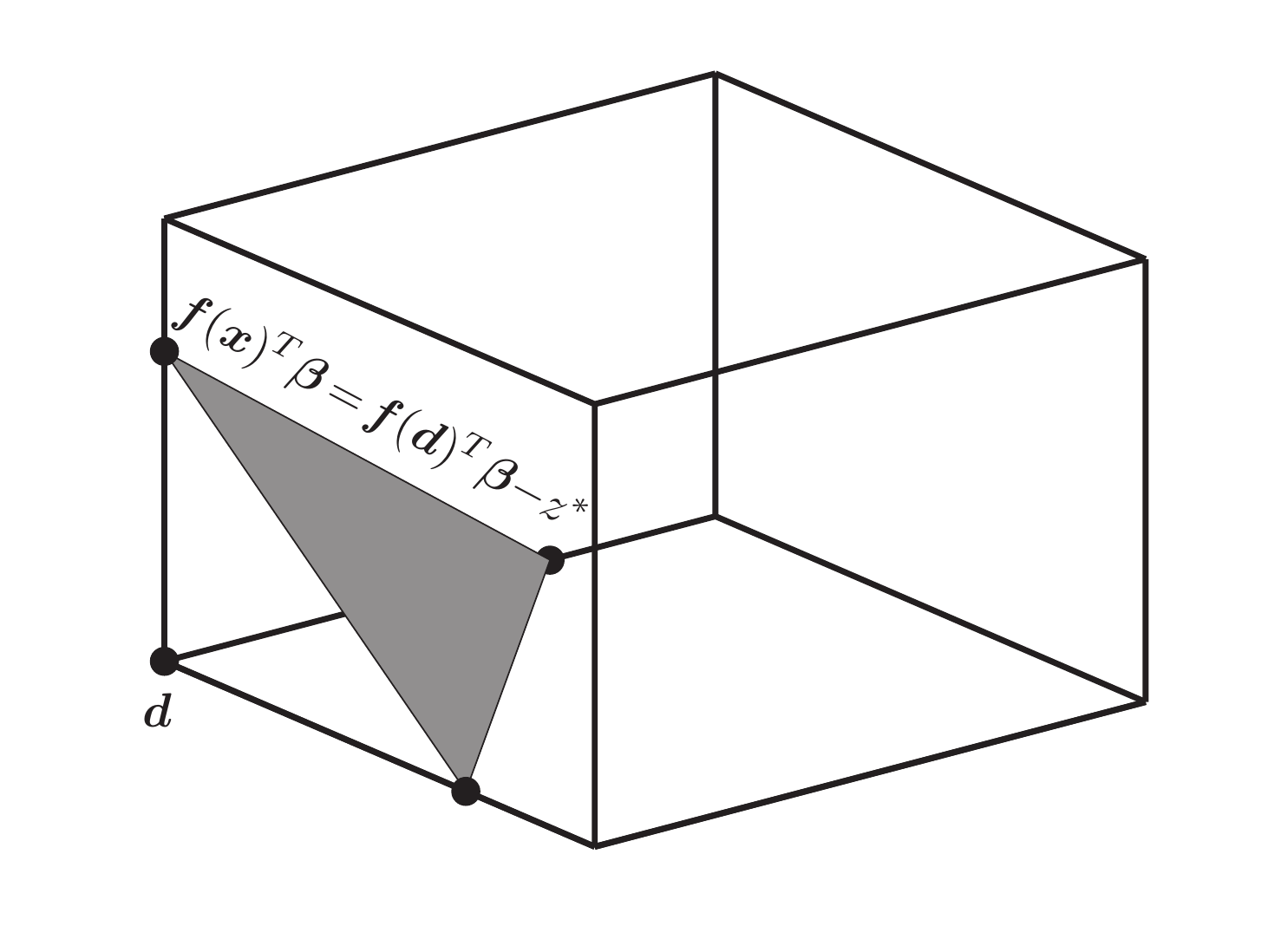}
\caption{Graphical representation of the $D$-optimal design for the Poisson-Gamma model for the case of three covariates}
\label{Abbildung D-optimales Design}
\end{figure}\\
The $D$-optimal design does not depend on $a$. By letting $a=b^{(p-1)/p}$ the $D$-optimality criterion function for the Poisson-Gamma model in \eqref{Determinante Informationsmatrix Poisson-Gamma-Modell} simplifies to $\det\bigl(\boldsymbol{M}(\xi;\boldsymbol{\beta})\bigr)=\det\bigl(\boldsymbol{M}_{\Po}(\xi;\boldsymbol{\beta})\bigr)/\big(b+m\cdot\boldsymbol{e}_1^T\boldsymbol{M}_{\Po}(\xi;\boldsymbol{\beta})\boldsymbol{e}_1\big)$. Hence, for $b\rightarrow0$ we obtain the $D_s$-optimal design for the parameters $\beta_1,\ldots,\beta_{p-1}$ for the Poisson model from Theorem \ref{Satz D-optimale Design Poisson-Gamma-Modell}.
\begin{Theorem}\label{Satz Ds-optimale Design Poisson-Modell}
Let $\mathscr{X}=\left[u_1,v_1\right]\times\ldots\times\left[u_{p-1},v_{p-1}\right]$ and $\beta_i\neq0$ for $i=1,\ldots,p-1$. Let $\boldsymbol{d}=(d_1,\ldots,d_{p-1})$ with $d_i=v_i$ for $\beta_i>0$ and $d_i=u_i$ for $\beta_i<0$. For any $z>0$ let
\begin{align}
w_p^{\ast}(z)=\frac{2}{p+\sqrt{(p-2)^2+4\cdot(p-1)\cdot e^{z}}}
\end{align}
and $w_1^{\ast}(z)=\ldots=w_{p-1}^{\ast}(z)=(1-w_p^{\ast}(z))/(p-1)$. The equation $z\cdot(1-w_p^{\ast}(z))=2$ has a unique solution $z^{\ast}$ in the interval $(0,\infty)$. If $z^{\ast}\leq\min_{i=1,\ldots,p-1}\bigl(\left|\beta_i\right|\cdot(v_i-u_i)\bigr)$, then the $D_s$-optimal design $\xi^{\ast}$ for $\beta_1,\ldots,\beta_{p-1}$ in the Poisson model is given by
\begin{align}
\xi^{\ast}=\begin{Bmatrix}\boldsymbol{d}-(z^{\ast}/\beta_1)\cdot \boldsymbol{e}_1 & \ldots & \boldsymbol{d}-(z^{\ast}/\beta_{p-1})\cdot \boldsymbol{e}_{p-1} & \boldsymbol{d}\\[2pt]w_1^{\ast}(z^{\ast}) & \ldots & w_{p-1}^{\ast}(z^{\ast}) & w_p^{\ast}(z^{\ast})\end{Bmatrix}.\label{Ds-optimale Design}
\end{align}
\end{Theorem}
The equation for $z^{\ast}$ can be written as $z=2/(1-w_p^{\ast}(z))$. Since $0<w_p^{\ast}(z)<1/p$, it follows that $2<z^{\ast}<2\cdot p/(p-1)$. In particular, in the Poisson model the distance of the support points on the edges to the support point at the vertex is larger for the $D_s$-optimal design than for the $D$-optimal design. Furthermore, in contrast to $D$-optimality, the $D_s$-optimal weights are not all equal.
\begin{Theorem}\label{Satz DA-optimale Design}
Let $\boldsymbol{f}=(1,f_1,\ldots,f_{p-1})^T$ and let $\boldsymbol{A}^T=\big(\boldsymbol{0}_s,\boldsymbol{\tilde{A}}\big)$, where $\boldsymbol{\tilde{A}}$ is a $s\times(p-1)$-matrix with $\rank(\boldsymbol{\tilde{A}})=s$. The $D_A$-optimality criterion functions for the Poisson and Poisson-Gamma model are identical. Hence, a design $\xi^{\ast}$ is $D_A$-optimal in the Poisson-Gamma model if and only if the design $\xi^{\ast}$ is $D_A$-optimal in the Poisson model.
\end{Theorem}
\begin{Remark}
Theorem \ref{Satz DA-optimale Design} holds for arbitrary regression functions $f_i$, $i=1,\ldots,p-1$. For the multiple regression model with $f_i(\boldsymbol{x})=x_i$ the $D_s$-optimal design for $\beta_1,\ldots,\beta_{p-1}$ in the Poisson-Gamma model is given by design \eqref{Ds-optimale Design} of Theorem \ref{Satz Ds-optimale Design Poisson-Modell}.
\end{Remark}
\begin{Example}
We consider the Poisson-Gamma and the Poisson model with one and two covariates. For the one covariate case, let $\mathscr{X}=\left[0,10\right]$, $\boldsymbol{f}(x)=(1,x)^T$ and $\boldsymbol{\beta}=(0,-1)^T$. For the two covariate case, let $\mathscr{X}=\left[0,10\right]\times\left[0,10\right]$, $\boldsymbol{f}(\boldsymbol{x})=(1,x_1,x_2)^T$ and $\boldsymbol{\beta}=(0,-1,-1)^T$. Furthermore, let $m=10$ and $b=1$.\\
The $D$-optimal designs for the Poisson-Gamma model follow from Theorem \ref{Satz D-optimale Design Poisson-Gamma-Modell}. The solution of equation \eqref{Gleichung Poisson-Gamma-Modell} is given by $z^{\ast}=2.341$ for one covariate and by $z^{\ast}=2.240$ for two covariates. For one covariate $D_s$-optimality for $\beta_1$ coincides with $c$-optimality with $\boldsymbol{c}=(0,1)^T$. Theorems \ref{Satz Ds-optimale Design Poisson-Modell} and \ref{Satz DA-optimale Design} yield the $D_s$-\,/\,$c$-optimal design for $\beta_1$ and the $D_s$-optimal design for $\beta_1$ and $\beta_2$ for the Poisson and Poisson-Gamma model. These optimal designs and the $D$-optimal design for the Poisson model are given in Tables \ref{Tabelle D-optimale Designs eine Kovariable Poisson-Gamma-Modell} and \ref{Tabelle D-optimale Designs Poisson-Gamma-Modell} and are additionally compared in terms of their efficiencies.
\begin{table}
\centering
\begin{tabular}{lccccc} 
\toprule
& & & \multicolumn{3}{c}{Efficiencies} \\ \cmidrule(r){4-6}
Model & Criterion & \hspace{26pt} Optimal design \hspace{26pt} & Po $D$ & P-G $D$ & Po $D_s$\\
\midrule
Poisson & $D$ & {\footnotesize $\begin{aligned}
\begin{Bmatrix}
0 & 2\\
1/2 & 1/2
\end{Bmatrix}
\end{aligned}$} & 1 & 0.925 & 0.769\\[8pt]
\addlinespace
Poisson-Gamma & $D$ & {\footnotesize $\begin{aligned}
\begin{Bmatrix}
0 & 2.341\\
0.297 & 0.703
\end{Bmatrix}
\end{aligned}$} & 0.902 & 1 & 0.974\\
\addlinespace
Poisson & \multirow{2}{*}{$D_s$\,/\,$c$} & \multirow{2}{*}{\footnotesize $\begin{aligned}
\begin{Bmatrix}
0 & 2.557\\
0.218 & 0.782
\end{Bmatrix}
\end{aligned}$} & \multirow{2}{*}{0.799} & \multirow{2}{*}{0.981} & \multirow{2}{*}{1}\\
Poisson-Gamma & & & & &\\
\addlinespace
\bottomrule
\end{tabular}
\caption{Optimal designs and comparison of efficiencies for the Poisson-Gamma model (P-G) and the Poisson model (Po) with one covariate and $\boldsymbol{\beta}=(0,-1)^T$}
\label{Tabelle D-optimale Designs eine Kovariable Poisson-Gamma-Modell}
\end{table}
\begin{table}
\centering
\begin{tabular}{lccccc} 
\toprule
& & & \multicolumn{3}{c}{Efficiencies} \\ \cmidrule(r){4-6}
Model & Criterion & Optimal design & Po $D$ & P-G $D$ & Po $D_s$\\ 
\midrule
Poisson & $D$ & {\footnotesize $\begin{aligned}
\begin{Bmatrix}
(2,0) & (0,2) & (0,0) \\
1/3 & 1/3 & 1/3
\end{Bmatrix}
\end{aligned}$} & 1 & 0.956 & 0.886\\[8pt]
\addlinespace
Poisson-Gamma & $D$ & {\footnotesize $\begin{aligned}
\begin{Bmatrix}
(2.240,0) & (0,2.240) & (0,0) \\
0.396 & 0.396 & 0.208
\end{Bmatrix}
\end{aligned}$} & 0.950 & 1 & 0.988\\
\addlinespace
Poisson & \multirow{2}{*}{$D_s$} & \multirow{2}{*}{\footnotesize $\begin{aligned}
\begin{Bmatrix}
(2.385,0) & (0,2.385) & (0,0) \\
0.419 & 0.419 & 0.162
\end{Bmatrix}
\end{aligned}$} & \multirow{2}{*}{0.895} & \multirow{2}{*}{0.990} & \multirow{2}{*}{1}\\
Poisson-Gamma & & & & &\\
\addlinespace
\bottomrule
\end{tabular}
\caption{Optimal designs and comparison of efficiencies for the Poisson-Gamma model (P-G) and the Poisson model (Po) with two covariates and $\boldsymbol{\beta}=(0,-1,-1)^T$}
\label{Tabelle D-optimale Designs Poisson-Gamma-Modell}
\end{table}\\
Both in the case of one covariate and in the case of two covariates all three optimal designs have a similar structure. One support point is located at the origin. The other support point(s) of the $D$-optimal design for the Poisson-Gamma model are located between the corresponding support point(s) of the $D$- and $D_s$-optimal design for the Poisson model.\\
In this example, compared to the $D$-optimal design in the Poisson model, the $D_s$-optimal design has a higher efficiency with respect to the $D$-optimal design in the Poisson-Gamma model. Furthermore, we observe that the efficiencies increase for all optimality criteria with the number of covariates, which can be explained by the designs getting more equal to each other. 
\end{Example}

\section{Optimal designs for linear optimality criteria} \label{Abschnitt Optimale Designs für lineare Optimalitätskriterien}

In this section let $\boldsymbol{f}=(1,f_1,\ldots,f_{p-1})^T$ with arbitrary regression functions $f_i$, for example $f_i(x)=x^i$ for polynomial regression with one-dimensional covariate $x\in\mathbb{R}$ or $f_i(\boldsymbol{x})=x_i$ for multiple regression with $\boldsymbol{x}\in\mathbb{R}^{p-1}$. Using Lemma \ref{Lemma Verallgemeinerte Inverse Informationsmatrix Poisson-Gamma-Modell} we show in the next theorem that the $c$- and $L$-optimal designs in the Poisson-Gamma model coincide with those in the Poisson model.
\begin{Theorem}\label{Satz L- und c-optimale Designs}
A design $\xi^{\ast}$ is $L$-optimal ($c$-optimal) in the Poisson-Gamma model if and only if the design $\xi^{\ast}$ is $L$-optimal ($c$-optimal) in the Poisson model.
\end{Theorem}
\begin{proof}
Let $\boldsymbol{B}=\boldsymbol{A}\boldsymbol{A}^T$. By Theorem \ref{Satz Identifizierbarkeit}, for a design $\xi$ the identifiability of $\boldsymbol{A}^T \boldsymbol{\beta}$ is equivalent in both models. With Lemma \ref{Lemma Verallgemeinerte Inverse Informationsmatrix Poisson-Gamma-Modell} for the relation between the generalized inverses of $\boldsymbol{M}(\xi;\boldsymbol{\beta})$ and $\boldsymbol{M}_{\Po}(\xi;\boldsymbol{\beta})$ we have for a design $\xi$:
\begin{align}
\tr\bigl(\boldsymbol{M}(\xi;\boldsymbol{\beta})^-\boldsymbol{B}\bigr)&=\tr\left(\left(\frac{b}{a}\cdot \boldsymbol{M}_{\Po}(\xi;\boldsymbol{\beta})^-+\frac{m}{a}\cdot \boldsymbol{e}_1\boldsymbol{e}_1^T\right)\cdot \boldsymbol{B}\right)\notag\\
&=\frac{b}{a}\cdot\tr\bigl(\boldsymbol{M}_{\Po}(\xi;\beta)^-\boldsymbol{B}\bigr)+\frac{m}{a}\cdot\boldsymbol{e}_1^T\boldsymbol{B}\boldsymbol{e}_1.\notag
\end{align}
Since the second summand does not depend on the design, the equivalence of the optimality of a design in both models follows.\\
With $\boldsymbol{B}=\boldsymbol{c}\boldsymbol{c}^T$ we obtain $c$-optimality as a special case of $L$-optimality.
\end{proof}
Theorem \ref{Satz L- und c-optimale Designs} shows that for the determination of $c$- and $L$-optimal designs we only have to consider the Poisson model, which facilitates the search for optimal designs. In particular, the results for the Poisson model of Ford et al.\ (1992) and Rodríguez-Torreblanca and Rodríguez-Díaz (2007) for one covariate and of Schmidt (2018) for multiple regression can be applied to the Poisson-Gamma model.

\section{Discussion}

For the Poisson-Gamma model the probability density function can be computed analytically, which allows deriving the information matrix. This is not possible for other distributions for the random effect like the normal distribution.\\
Based on some relations between the information matrices of the Poisson and Poisson-Gamma model we decomposed the $D$-optimality criterion function for the Poisson-Gamma model into a weighted sum of the $D$- and $D_s$-optimality criterion function for the Poisson model. The optimal designs for all these optimality criteria have the same structure, in particular they have a minimal support. Apart from the differences concerning the support points, the $D$-optimal weights for the Poisson-Gamma model differ from the equal allocation rule of the $D$-optimal weights for the Poisson model.\\
Since for $L$- and $c$-optimality the optimal designs are equal for the Poisson and Poisson-Gamma model, known results for the Poisson model can be used.\\
A possible extension of this work is to investigate for other optimality criteria, if there is also some relation between the optimal designs for the Poisson and Poisson-Gamma model. Since both models are nonlinear, the optimal designs depend on the unknown parameters for all optimality criteria under consideration. A way to obtain more robust designs regarding parameter misspecification is to use standardized maximin optimality criteria (cf.\ Müller, 1995), which maximize the worst efficiency with respect to a prespecified parameter set.

\appendix
\section{Appendix}

\begin{proof}[Proof of Theorem \ref{Satz Dichte Poisson-Gamma-Modell}]\ \\
Let $\boldsymbol{y}=(y_1,\ldots,y_m)$. Integration over the random effect yields:
\begin{align}
f_{\boldsymbol{Y}}(\boldsymbol{y})&=\int_0^{\infty} f_{\boldsymbol{Y}|\Theta=\theta}(\boldsymbol{y})\cdot f_{\Theta}(\theta)\thinspace\mathrm{d}\theta\notag\\[-4pt]
&=\int_0^{\infty} \prod_{j=1}^m \left(\theta^{y_j}\cdot\frac{e^{\boldsymbol{f}(\boldsymbol{x}_j)^T\boldsymbol{\beta}\cdot y_j}}{y_j!}\cdot e^{-\theta\cdot e^{\boldsymbol{f}(\boldsymbol{x}_j)^T\boldsymbol{\beta}}}\right)\cdot\frac{b^a}{\Gamma(a)}\cdot\theta^{a-1}\cdot e^{-b\cdot\theta}\thinspace\mathrm{d}\theta\notag\\[-4pt]
&=\int_0^{\infty} \theta^{\sum_{j=1}^m y_j}\cdot\frac{e^{\sum_{j=1}^m \boldsymbol{f}(\boldsymbol{x}_j)^T\beta\cdot y_j}}{\prod_{j=1}^m y_j!}\cdot e^{-\theta\cdot \sum_{j=1}^m e^{\boldsymbol{f}(\boldsymbol{x}_j)^T\boldsymbol{\beta}}}\cdot\frac{b^a}{\Gamma(a)}\cdot\theta^{a-1}\cdot e^{-b\cdot\theta}\thinspace\mathrm{d}\theta.\notag
\end{align}
With $\tilde{a}=a+\sum_{j=1}^m y_j$ and $\tilde{b}=b+\sum_{j=1}^m e^{\boldsymbol{f}(\boldsymbol{x}_j)^T\boldsymbol{\beta}}$ we obtain
\begin{align}
f_{\boldsymbol{Y}}(\boldsymbol{y})&=\frac{b^a\cdot e^{\sum_{j=1}^m \boldsymbol{f}(\boldsymbol{x}_j)^T\boldsymbol{\beta}\cdot y_j}\cdot\Gamma(\tilde{a})}{\Gamma(a)\cdot\prod_{j=1}^m y_j!\cdot\tilde{b}^{\tilde{a}}}\cdot\int_0^{\infty} \frac{\tilde{b}^{\tilde{a}}}{\Gamma(\tilde{a})}\cdot\theta^{\tilde{a}-1}\cdot e^{-\tilde{b}\cdot\theta}\thinspace\mathrm{d}\theta=\frac{b^a\cdot e^{\sum_{j=1}^m \boldsymbol{f}(\boldsymbol{x}_j)^T\boldsymbol{\beta}\cdot y_j}\cdot\Gamma(\tilde{a})}{\Gamma(a)\cdot\prod_{j=1}^m y_j!\cdot\tilde{b}^{\tilde{a}}}\notag
\end{align}
as the joint density of $Y_1,\ldots,Y_m$.
\end{proof}
\begin{proof}[Proof of Theorem \ref{Satz Fisher-Informationsmatrix Poisson-Gamma-Modell}]\ \\
The logarithm of the probability density function of $\boldsymbol{Y}=(Y_1,\ldots,Y_m)$, which was derived in Theorem \ref{Satz Dichte Poisson-Gamma-Modell}, is given by
\begin{align}
\ln\bigl(f_{\boldsymbol{Y}}(\boldsymbol{y})\bigr)&=C(\boldsymbol{y})+\sum_{j=1}^m \boldsymbol{f}(\boldsymbol{x}_j)^T\boldsymbol{\beta}\cdot y_j-\left(a+\sum_{j=1}^m y_j\right)\cdot\ln\left(b+\sum_{j=1}^m e^{\boldsymbol{f}(\boldsymbol{x}_j)^T\boldsymbol{\beta}}\right)\notag
\end{align}
with $C(\boldsymbol{y})$ not depending on $\boldsymbol{\beta}$.
The first and second derivative with respect to $\boldsymbol{\beta}$ are given by:
\begin{align}
\frac{\partial\ln\bigl(f_{\boldsymbol{Y}}(\boldsymbol{y})\bigr)}{\partial\boldsymbol{\beta}}&=\sum_{j=1}^m \boldsymbol{f}(\boldsymbol{x}_j)^T\cdot y_j-\frac{a+\sum_{j=1}^m y_j}{b+\sum_{j=1}^m e^{\boldsymbol{f}(\boldsymbol{x}_j)^T\boldsymbol{\beta}}}\cdot\sum_{j=1}^m e^{\boldsymbol{f}(\boldsymbol{x}_j)^T\boldsymbol{\beta}}\cdot\boldsymbol{f}(\boldsymbol{x}_j)^T,\notag\\
\frac{\partial^2\ln\bigl(f_{\boldsymbol{Y}}(\boldsymbol{y})\bigr)}{\partial\boldsymbol{\beta}\thinspace\partial\boldsymbol{\beta}^T}&=\frac{a+\sum_{j=1}^m y_j}{\big(b+\sum_{j=1}^m e^{\boldsymbol{f}(\boldsymbol{x}_j)^T\boldsymbol{\beta}}\big)^2}\cdot\sum_{j=1}^m e^{\boldsymbol{f}(\boldsymbol{x}_j)^T\boldsymbol{\beta}}\cdot\boldsymbol{f}(\boldsymbol{x}_j)\cdot\sum_{j=1}^m e^{\boldsymbol{f}(\boldsymbol{x}_j)^T\boldsymbol{\beta}}\cdot\boldsymbol{f}(\boldsymbol{x}_j)^T\notag\\&\;\;\;\;-\frac{a+\sum_{j=1}^m y_j}{b+\sum_{j=1}^m e^{\boldsymbol{f}(\boldsymbol{x}_j)^T\boldsymbol{\beta}}}\cdot\sum_{j=1}^m e^{\boldsymbol{f}(\boldsymbol{x}_j)^T\beta}\cdot\boldsymbol{f}(\boldsymbol{x}_j)\boldsymbol{f}(\boldsymbol{x}_j)^T.\notag
\end{align}
Since $\boldsymbol{I}(\boldsymbol{\beta})=-E\left(\frac{\partial^2\ln(f_{\boldsymbol{Y}}(\boldsymbol{Y}))}{\partial\boldsymbol{\beta}\thinspace\partial\boldsymbol{\beta}^T}\right)$ and $E(Y_j)=\frac{a}{b}\cdot\exp(\boldsymbol{f}(\boldsymbol{x}_j)^T\boldsymbol{\beta})$ the Fisher information matrix for the parameter vector $\boldsymbol{\beta}$ is given by:
\begin{align}
\boldsymbol{I}(\boldsymbol{\beta})=\frac{a}{b}\cdot\left(\sum_{j=1}^m e^{\boldsymbol{f}(\boldsymbol{x}_j)^T\boldsymbol{\beta}}\negthinspace\cdot\negthinspace \boldsymbol{f}(x_j)\boldsymbol{f}(x_j)^T-\frac{\sum_{j=1}^m e^{\boldsymbol{f}(\boldsymbol{x}_j)^T\boldsymbol{\beta}}\negthinspace\cdot\negthinspace \boldsymbol{f}(x_j)\cdot\sum_{j=1}^m e^{\boldsymbol{f}(\boldsymbol{x}_j)^T\boldsymbol{\beta}}\negthinspace\cdot\negthinspace \boldsymbol{f}(x_j)^T}{\sum_{j=1}^m e^{\boldsymbol{f}(\boldsymbol{x}_j)^T\boldsymbol{\beta}}+b}\right).\notag
\end{align}
By identifying the components of $\boldsymbol{I}(\boldsymbol{\beta})$ with those of $\boldsymbol{I}_{\Po}(\boldsymbol{\beta})$, the information matrix $\boldsymbol{I}(\boldsymbol{\beta})$ can be represented in terms of $\boldsymbol{I}_{\Po}(\boldsymbol{\beta})$ as in Theorem \ref{Satz Fisher-Informationsmatrix Poisson-Gamma-Modell}.
\end{proof}
\begin{Lemma}\label{Lemma Matrix für Darstellung Informationsmatrix Poisson-Gamma-Modell}
Let the matrix $\boldsymbol{L}(\xi;\boldsymbol{\beta})$ be given by
\begin{align}
\boldsymbol{L}(\xi;\boldsymbol{\beta})=\boldsymbol{I}-\frac{\boldsymbol{M}_{\Po}(\xi;\boldsymbol{\beta})\boldsymbol{e}_1\boldsymbol{e}_1^T}{\boldsymbol{e}_1^T\boldsymbol{M}_{\Po}(\xi;\boldsymbol{\beta})\boldsymbol{e}_1+\frac{b}{m}}.
\end{align}
Then the following statements hold:
\begin{enumerate}
\item[(i)] $\boldsymbol{M}(\xi;\boldsymbol{\beta})=\frac{a}{b}\cdot\boldsymbol{L}(\xi;\boldsymbol{\beta})\boldsymbol{M}_{\Po}(\xi;\boldsymbol{\beta})=\frac{a}{b}\cdot\boldsymbol{M}_{\Po}(\xi;\boldsymbol{\beta})\boldsymbol{L}(\xi;\boldsymbol{\beta})^T$.
\item[(ii)] The matrix $\boldsymbol{L}(\xi;\boldsymbol{\beta})$ is regular.
\item[(iii)] $\boldsymbol{L}(\xi;\boldsymbol{\beta})=\boldsymbol{I}-\frac{m}{a}\cdot\boldsymbol{M}(\xi;\boldsymbol{\beta})\boldsymbol{e}_1\boldsymbol{e}_1^T$.
\end{enumerate}
\end{Lemma}
\begin{proof}\ \\
Let $\boldsymbol{M}=\boldsymbol{M}(\xi;\boldsymbol{\beta})$, $\boldsymbol{M}_{\Po}=\boldsymbol{M}_{\Po}(\xi;\boldsymbol{\beta})$ and $\boldsymbol{L}=\boldsymbol{L}(\xi;\boldsymbol{\beta})$.\\
(i) The statement follows directly from equation \eqref{Darstellung Informationsmatrix Poisson-Gamma-Modell} for the information matrix, where for the second equation the symmetry of $\boldsymbol{M}_{\Po}$ is used.\\
(ii) $\boldsymbol{L}$ is a lower triangular matrix in which all main diagonal entries are equal to one except for the top diagonal entry. Since $\det(\boldsymbol{L})=1-\left(\boldsymbol{e}_1^T\boldsymbol{M}_{\Po}\boldsymbol{e}_1\right)/\left(\boldsymbol{e}_1^T\boldsymbol{M}_{\Po}\boldsymbol{e}_1+\frac{b}{m}\right)\neq0$ the matrix $\boldsymbol{L}$ is regular.\\
(iii) The statement follows with equation \eqref{Darstellung Informationsmatrix Poisson-Gamma-Modell}:
\begin{align}
\boldsymbol{I}-\frac{m}{a}\cdot\boldsymbol{M}\boldsymbol{e}_1\boldsymbol{e}_1^T&=\boldsymbol{I}-\frac{m}{b}\cdot\left(\boldsymbol{M}_{\Po}-\frac{\boldsymbol{M}_{\Po}\boldsymbol{e}_1\boldsymbol{e}_1^T\boldsymbol{M}_{\Po}}{\boldsymbol{e}_1^T\boldsymbol{M}_{\Po}\boldsymbol{e}_1+\frac{b}{m}}\right)\boldsymbol{e}_1\boldsymbol{e}_1^T\notag\\
&=\boldsymbol{I}-\frac{m}{b}\cdot\boldsymbol{M}_{\Po}\boldsymbol{e}_1\boldsymbol{e}_1^T\cdot\left(1-\frac{\boldsymbol{e}_1^T\boldsymbol{M}_{\Po}\boldsymbol{e}_1}{\boldsymbol{e}_1^T\boldsymbol{M}_{\Po}\boldsymbol{e}_1+\frac{b}{m}}\right)=\boldsymbol{I}-\frac{\boldsymbol{M}_{\Po}\boldsymbol{e}_1\boldsymbol{e}_1^T}{\boldsymbol{e}_1^T\boldsymbol{M}_{\Po}\boldsymbol{e}_1+\frac{b}{m}}.\notag
\end{align}
Here it was used that $\boldsymbol{e}_1^T\boldsymbol{M}_{\Po}\boldsymbol{e}_1$ is a real number.
\end{proof}
\begin{proof}[Proof of Lemma \ref{Lemma Rang Zusammenhang Informationsmatrix}]\ \\
By Lemma \ref{Lemma Matrix für Darstellung Informationsmatrix Poisson-Gamma-Modell} the information matrix for the Poisson-Gamma model can be written as $\boldsymbol{M}(\xi;\boldsymbol{\beta})=\frac{a}{b}\cdot\boldsymbol{L}(\xi;\boldsymbol{\beta})\boldsymbol{M}_{\Po}(\xi;\boldsymbol{\beta})$, where $\boldsymbol{L}(\xi;\boldsymbol{\beta})$ is a regular matrix and thus has full rank. It follows that $\rank(\boldsymbol{M}(\xi;\boldsymbol{\beta}))=\rank\bigl(\boldsymbol{L}(\xi;\boldsymbol{\beta})\boldsymbol{M}_{\Po}(\xi;\boldsymbol{\beta})\bigr)=\rank(\boldsymbol{M}_{\Po}(\xi;\boldsymbol{\beta}))$.
\end{proof}
\begin{proof}[Proof of Theorem \ref{Satz Identifizierbarkeit}]\ \\
Let $\boldsymbol{H}$ be a $p\times s$-matrix and let $\tilde{\boldsymbol{H}}=\frac{b}{a}\cdot(\boldsymbol{L}(\xi;\boldsymbol{\beta})^T)^{-1}\boldsymbol{H}$ with the regular matrix $\boldsymbol{L}(\xi;\boldsymbol{\beta})$ from Lemma \ref{Lemma Matrix für Darstellung Informationsmatrix Poisson-Gamma-Modell}. We have with Lemma \ref{Lemma Matrix für Darstellung Informationsmatrix Poisson-Gamma-Modell} (i):
\begin{align}
\boldsymbol{A}-\boldsymbol{M}(\xi;\boldsymbol{\beta})\tilde{\boldsymbol{H}}&=\boldsymbol{A}-\boldsymbol{M}_{\Po}(\xi;\boldsymbol{\beta})\boldsymbol{L}(\xi;\boldsymbol{\beta})^T(\boldsymbol{L}(\xi;\boldsymbol{\beta})^T)^{-1}\boldsymbol{H}=\boldsymbol{A}-\boldsymbol{M}_{\Po}(\xi;\boldsymbol{\beta})\boldsymbol{H}.\notag
\end{align}
Hence, $\boldsymbol{A}=\boldsymbol{M}(\xi;\boldsymbol{\beta})\tilde{\boldsymbol{H}}$ is equivalent to $\boldsymbol{A}=\boldsymbol{M}_{\Po}(\xi;\boldsymbol{\beta})\boldsymbol{H}$ and thus the identifiability of $\boldsymbol{A}^T \boldsymbol{\beta}$ is equivalent in both models.
\end{proof}
\begin{proof}[Proof of Lemma \ref{Lemma Verallgemeinerte Inverse Informationsmatrix Poisson-Gamma-Modell}]\ \\
Let $\boldsymbol{M}=\boldsymbol{M}(\xi;\boldsymbol{\beta})$ and $\boldsymbol{M}_{\Po}=\boldsymbol{M}_{\Po}(\xi;\boldsymbol{\beta})$. Let $\boldsymbol{M}_{\Po}^-$ be a $p\times p$-matrix and let $\boldsymbol{M}^-=\frac{b}{a}\cdot \boldsymbol{M}_{\Po}^-+\frac{m}{a}\cdot \boldsymbol{e}_1\boldsymbol{e}_1^T$. The matrix $\boldsymbol{M}^-$ is a generalized inverse of $\boldsymbol{M}$ if it satisfies the condition $\boldsymbol{M}\boldsymbol{M}^-\boldsymbol{M}=\boldsymbol{M}$. We have:
\begin{align}
\boldsymbol{M}\boldsymbol{M}^-\boldsymbol{M}=\boldsymbol{M}\cdot\left(\frac{b}{a}\cdot \boldsymbol{M}_{\Po}^-+\frac{m}{a}\cdot \boldsymbol{e}_1\boldsymbol{e}_1^T\right)\cdot \boldsymbol{M}=\frac{b}{a}\cdot\boldsymbol{M}\boldsymbol{M}_{\Po}^-\boldsymbol{M}+\frac{m}{a}\cdot\boldsymbol{M}\boldsymbol{e}_1\boldsymbol{e}_1^T\boldsymbol{M}.\notag
\end{align}
With the matrix $\boldsymbol{L}=\boldsymbol{L}(\xi;\boldsymbol{\beta})$ from Lemma \ref{Lemma Matrix für Darstellung Informationsmatrix Poisson-Gamma-Modell} it follows that $(b/a)\cdot \boldsymbol{M}\boldsymbol{M}_{\Po}^-\boldsymbol{M}=(a/b)\cdot\boldsymbol{L}\boldsymbol{M}_{\Po}\boldsymbol{M}_{\Po}^-\boldsymbol{M}_{\Po}\boldsymbol{L}^T$. Furthermore, by Lemma \ref{Lemma Matrix für Darstellung Informationsmatrix Poisson-Gamma-Modell} (i) and (iii) we have:
\begin{align}
\frac{m}{a}\cdot\boldsymbol{M}\boldsymbol{e}_1\boldsymbol{e}_1^T\boldsymbol{M}-\boldsymbol{M}=\left(\frac{m}{a}\cdot\boldsymbol{M}\boldsymbol{e}_1\boldsymbol{e}_1^T-\boldsymbol{I}\right)\cdot\boldsymbol{M}=-\boldsymbol{L}\boldsymbol{M}=-\frac{a}{b}\cdot\boldsymbol{L}\boldsymbol{M}_{\Po}\boldsymbol{L}^T.\notag
\end{align}
Hence $\boldsymbol{M}\boldsymbol{M}^-\boldsymbol{M}-\boldsymbol{M}=(a/b)\cdot\boldsymbol{L}\cdot\left(\boldsymbol{M}_{\Po}\boldsymbol{M}_{\Po}^-\boldsymbol{M}_{\Po}-\boldsymbol{M}_{\Po}\right)\cdot\boldsymbol{L}^T$. The matrix $\boldsymbol{L}$ is regular by Lemma \ref{Lemma Matrix für Darstellung Informationsmatrix Poisson-Gamma-Modell} (ii). Thus $\boldsymbol{M}\boldsymbol{M}^-\boldsymbol{M}=\boldsymbol{M}$ is equivalent to $\boldsymbol{M}_{\Po}\boldsymbol{M}_{\Po}^-\boldsymbol{M}_{\Po}=\boldsymbol{M}_{\Po}$.
\end{proof}
\begin{proof}[Proof of Theorem \ref{Satz Kriteriumsfunktion D-Optimalität Poisson-Gamma-Modell}]\ \\
If $\boldsymbol{M}(\xi;\boldsymbol{\beta})$ is singular, then so is $\boldsymbol{M}_{\Po}(\xi;\boldsymbol{\beta})$ according to Lemma \ref{Lemma Rang Zusammenhang Informationsmatrix} and thus equation \eqref{Determinante Informationsmatrix Poisson-Gamma-Modell} holds. If $\boldsymbol{M}(\xi;\boldsymbol{\beta})$ is regular, we have with the inverse of $\boldsymbol{M}(\xi;\boldsymbol{\beta})$ from Remark~\ref{Bemerkung Inverse Informationsmatrix Poisson-Gamma-Modell}:
\begin{align}
\det\bigl(\boldsymbol{M}(\xi;\boldsymbol{\beta})\bigr)&=\frac{1}{\det\bigl(\boldsymbol{M}(\xi;\boldsymbol{\beta})^{-1}\bigr)}=\frac{1}{\det\left(\frac{b}{a}\cdot \boldsymbol{M}_{\Po}(\xi;\boldsymbol{\beta})^{-1}+\frac{m}{a}\cdot \boldsymbol{e}_1\boldsymbol{e}_1^T\right)}\notag\\
&=\frac{1}{\left(1+\frac{a}{b}\cdot\frac{m}{a}\cdot\boldsymbol{e}_1^T\boldsymbol{M}_{\Po}(\xi;\boldsymbol{\beta})\boldsymbol{e}_1\right)\cdot\det\left(\frac{b}{a}\cdot \boldsymbol{M}_{\Po}(\xi;\boldsymbol{\beta})^{-1}\right)}\notag\\
&=\left(\frac{a}{b}\right)^p\cdot\frac{\det\bigl(\boldsymbol{M}_{\Po}(\xi;\boldsymbol{\beta})\bigr)}{1+\frac{m}{b}\cdot\boldsymbol{e}_1^T\boldsymbol{M}_{\Po}(\xi;\boldsymbol{\beta})\boldsymbol{e}_1}.\notag
\end{align}
We note that in the third step the matrix determinant lemma was used.
\end{proof}
\begin{proof}[Proof of Theorem \ref{Satz D-optimale Gewichte Poisson-Gamma-Modell}]\ \\
Let $a_j=1+\frac{m}{b}\cdot\exp\bigl(\boldsymbol{f}(\boldsymbol{x}_j)^T\boldsymbol{\beta}\bigr)$ for $j=1,\ldots,p$. By Theorem \ref{Satz Kriteriumsfunktion D-Optimalität Poisson-Gamma-Modell} and with $\boldsymbol{M}_{\Po}(\xi;\boldsymbol{\beta})=\boldsymbol{X}^T\boldsymbol{W}\boldsymbol{\Lambda}\boldsymbol{X}$ the criterion function is given by:
\begin{align}
\det(\boldsymbol{M}(\xi;\boldsymbol{\beta}))=\left(\frac{a}{b}\right)^p\cdot\frac{\det(\boldsymbol{M}_{\Po}(\xi;\boldsymbol{\beta}))}{1+\frac{m}{b}\cdot\sum_{j=1}^{p} w_j\cdot e^{\boldsymbol{f}(\boldsymbol{x}_j)^T\boldsymbol{\beta}}}=\left(\frac{a}{b}\right)^p\cdot\frac{\prod_{j=1}^p w_j\cdot\det(\boldsymbol{X})^2\cdot\det(\boldsymbol{\Lambda})}{\sum_{j=1}^p w_j\cdot a_j}.\notag
\end{align}
Since the support points $\boldsymbol{x}_1,\ldots,\boldsymbol{x}_{p-1}$ satisfy $\boldsymbol{f}(\boldsymbol{x})^T\boldsymbol{\beta}=c$, we have $a_1=\ldots=a_{p-1}$. Hence, the function
\begin{align}
\frac{\prod_{j=1}^p w_j}{\sum_{j=1}^p w_j\cdot a_j}=\frac{\prod_{j=1}^p w_j}{a_1\cdot\sum_{j=1}^{p-1} w_1+w_p\cdot a_p}=\frac{\prod_{j=1}^p w_j}{a_1\cdot(1-w_p)+w_p\cdot a_p}\notag
\end{align}
is to be maximized with respect to the weights $w_1,\ldots,w_p$. For fixed $w_p$ the product is maximal for $w_1=\ldots=w_{p-1}$. Thus, the optimization problem simplifies to maximising the expression $(w_1^{p-1}\cdot w_p)/\left[(1-w_p)\cdot a_1+w_p\cdot a_p\right]$. With $w_1=(1-w_p)/(p-1)$ we have to maximize the function
\begin{align}
g(w_p)=\frac{1}{(p-1)^{p-1}}\cdot\frac{(1-w_p)^{p-1}\cdot w_p}{(1-w_p)\cdot a_1+w_p\cdot a_p}\notag
\end{align}
with respect to $w_p$. Setting the first derivative of $g$ equal to zero yields:
\begin{align}
0=w_p^2\cdot(p-1)\cdot(a_1-a_p)-w_p\cdot p\cdot a_1+a_1.\label{D-optimale Gewichte Ableitung Gleichung Poisson-Gamma-Modell}
\end{align}
This quadratic equation has one solution in the interval $(0,1)$, which is given by
$w_p^{\ast}=2/\big(p+\sqrt{(p-2)^2+4\cdot(p-1)\cdot(a_p/a_1)}\thinspace\big)$. We have $g(0)=g(1)=0$ and $g(w_p)>0$ for $w_p\in(0,1)$. Hence, in the interval $(0,1)$ the function $g$ is maximal at $w_p^{\ast}$. It follows that $w_1^{\ast}=(1-w_p^{\ast})/(p-1)$. Because of $a_p>a_1$ we have:
\begin{align}
w_p^{\ast}<\frac{2}{p+\sqrt{(p-2)^2+4\cdot(p-1)}}=\frac{2}{p+\sqrt{p^2}}=\frac{1}{p}.\notag
\end{align}
With $1=(p-1)\cdot w_1^{\ast}+w_p^{\ast}<(p-1)\cdot w_1^{\ast}+\frac{1}{p}$ and $1=(p-1)\cdot w_1^{\ast}+w_p^{\ast}>(p-1)\cdot w_1^{\ast}$ it follows that $\frac{1}{p}<w_1^{\ast}<\frac{1}{p-1}$.
\end{proof}
\begin{Lemma}\label{Lemma D-optimale Gewichte Gleichung Poisson-Gamma-Modell}
Let $\xi$ be a design with support points $\boldsymbol{x}_1,\ldots,\boldsymbol{x}_p$ and $D$-optimal weights $w_1^{\ast},\ldots,w_p^{\ast}$ as in Theorem \ref{Satz D-optimale Gewichte Poisson-Gamma-Modell}. Then the following equation holds:
\begin{align}
\boldsymbol{e}_1^T\boldsymbol{M}_{\Po}(\xi;\boldsymbol{\beta})\boldsymbol{e}_1\cdot w_p^{\ast}-w_1^{\ast}\cdot e^c+\frac{b}{m}\cdot(1-p\cdot w_1^{\ast})=0.\notag
\end{align}
\end{Lemma}
\begin{proof}\ \\
Let $\lambda_1=\exp(c)$ and $\lambda_p=\exp(\boldsymbol{f}(\boldsymbol{x}_p)^T\boldsymbol{\beta})$. According to equation \eqref{D-optimale Gewichte Ableitung Gleichung Poisson-Gamma-Modell} from the proof of Theorem \ref{Satz D-optimale Gewichte Poisson-Gamma-Modell} the $D$-optimal weights satisfy $0=(w_p^{\ast})^2\cdot(p-1)\cdot(a_1-a_p)-w_p^{\ast}\cdot p\cdot a_1+a_1$, where $a_1=1+\frac{m}{b}\cdot \lambda_1$ and $a_p=1+\frac{m}{b}\cdot \lambda_p$. It follows that:
\begin{align}
0=\frac{m}{b}\cdot(w_p^{\ast})^2\cdot(p-1)\cdot(\lambda_1-\lambda_p)+\left(1+\frac{m}{b}\cdot \lambda_1\right)\cdot(1-p\cdot w_p^{\ast}).\notag
\end{align}
We have $w_p^{\ast}=1-(p-1)\cdot w_1^{\ast}$. Multiplication by $b/m$ and division by $p-1$ yields:
\begin{align}
0&=w_p^{\ast}\cdot\big(1-(p-1)\cdot w_1^{\ast}\big)\cdot \lambda_1-(w_p^{\ast})^2\cdot \lambda_p+\left(\frac{b}{m}+\lambda_1\right)\cdot\frac{1-p\cdot w_p^{\ast}}{p-1}\notag\\
&=-w_p^{\ast}\cdot\big((p-1)\cdot w_1^{\ast}\cdot \lambda_1+w_p^{\ast}\cdot \lambda_p\big)+\left(\frac{1-w_p^{\ast}}{p-1}\right)\cdot \lambda_1+\frac{b}{m}\cdot\frac{1-p\cdot w_p^{\ast}}{p-1}\notag\\
&=-w_p^{\ast}\cdot \boldsymbol{e}_1^T\boldsymbol{M}_{\Po}(\xi;\boldsymbol{\beta})\boldsymbol{e}_1+w_1^{\ast}\cdot \lambda_1+\frac{b}{m}\cdot\frac{1-p\cdot w_p^{\ast}}{p-1}.\notag
\end{align}
With $(1-p\cdot w_p^{\ast})/(p-1)=\big(1-p\cdot(1-(p-1)\cdot w_1^{\ast})\big)/(p-1)=-1+p\cdot w_1^{\ast}$ and multiplication with $-1$ the equation given in the lemma follows.
\end{proof}
\begin{proof}[Proof of Theorem \ref{Satz D-optimale Design Poisson-Gamma-Modell}]\ \\
For simplicity of notation let $\boldsymbol{M}=\boldsymbol{M}(\xi^{\ast},\boldsymbol{\beta})$ and $\boldsymbol{M}_{\Po}=\boldsymbol{M}_{\Po}(\xi^{\ast},\boldsymbol{\beta})$. Since the optimization problem does not depend on the positive constant $a$, we choose $a=b$.\\
The extended design region $\mathscr{X}_{\text{ext}}$ with $x_i\in(-\infty,v_i]$ for $\beta_i>0$ and $x_i\in[u_i,\infty)$ for $\beta_i<0$ is considered. Since the matrix $\boldsymbol{M}_{\Po}^{-1}\boldsymbol{M}\boldsymbol{M}_{\Po}^{-1}$ is positive definite and the sets $\left\{\boldsymbol{x}\in\mathbb{R}^{p-1}:\boldsymbol{f}(\boldsymbol{x})^T\boldsymbol{\beta}=c\right\}\cap\mathscr{X}_{\text{ext}}$ are bounded for all $c\in\mathbb{R}$, the left-hand side of the condition in the Equivalence Theorem \ref{Satz Äquivalenzsatz} is maximized at the edges of the extended design region $\mathscr{X}_{\text{ext}}$, so it suffices to show that the condition of the equivalence theorem is satisfied on the edges of $\mathscr{X}_{\text{ext}}$ in order to prove the $D$-optimality of the design $\xi^{\ast}$ (cf.\ Schmidt and Schwabe (2017), Schmidt (2018)). Thus, the design $\xi^{\ast}$ is $D$-optimal if
\begin{align}
e^{\boldsymbol{f}(\boldsymbol{d}+(x_i-d_i)\cdot \boldsymbol{e}_i)^T\boldsymbol{\beta}}\cdot \boldsymbol{f}\bigl(\boldsymbol{d}+(x_i-d_i)\cdot \boldsymbol{e}_i\bigr)^T\boldsymbol{M}_{\Po}^{-1}\boldsymbol{M}\boldsymbol{M}_{\Po}^{-1}\boldsymbol{f}\bigl(\boldsymbol{d}+(x_i-d_i)\cdot \boldsymbol{e}_i\bigr)\leq\tr\bigl(\boldsymbol{M}\boldsymbol{M}_{\Po}^{-1}\bigr)\notag
\end{align}
holds for all $x_i$, $i=1,\ldots,p-1$. We define the following two functions:
\begin{align}
g_i(x)&=\boldsymbol{f}\bigl(\boldsymbol{d}+(x-d_i)\cdot \boldsymbol{e}_i\bigr)^T\boldsymbol{M}_{\Po}^{-1}\boldsymbol{M}\boldsymbol{M}_{\Po}^{-1}\boldsymbol{f}\bigl(\boldsymbol{d}+(x-d_i)\cdot \boldsymbol{e}_i\bigr),\notag\\
h_i(x)&=g_i(x)-\exp\bigl(-\boldsymbol{f}(\boldsymbol{d}+(x-d_i)\cdot \boldsymbol{e}_i)^T\boldsymbol{\beta}\bigr)\cdot\tr\bigl(\boldsymbol{M}\boldsymbol{M}_{\Po}^{-1}\bigr).\notag
\end{align}
The condition of the equivalence theorem is equivalent to
\begin{align}
h_i(x_i)=g_i(x_i)-\exp\bigl(-\boldsymbol{f}(\boldsymbol{d}+(x_i-d_i)\cdot \boldsymbol{e}_i)^T\boldsymbol{\beta}\bigr)\cdot\tr\bigl(\boldsymbol{M}\boldsymbol{M}_{\Po}^{-1}\bigr)\leq0\label{Äquivalenzsatz Bedingung}
\end{align}
for all $x_i$, $i=1,\ldots,p-1$.\\
With the matrix $\boldsymbol{L}=\boldsymbol{I}-\left(\boldsymbol{M}_{\Po}\boldsymbol{e}_{1,p}\boldsymbol{e}_{1,p}^T\right)/\left(\boldsymbol{e}_{1,p}^T\boldsymbol{M}_{\Po}\boldsymbol{e}_{1,p}+\frac{b}{m}\right)$
from Lemma \ref{Lemma Matrix für Darstellung Informationsmatrix Poisson-Gamma-Modell} we have $\boldsymbol{M}\boldsymbol{M}_{\Po}^{-1}=\boldsymbol{L}\boldsymbol{M}_{\Po}\boldsymbol{M}_{\Po}^{-1}=\boldsymbol{L}$ and hence:
\begin{align}
\boldsymbol{M}_{\Po}^{-1}\boldsymbol{M}\boldsymbol{M}_{\Po}^{-1}&=\boldsymbol{M}_{\Po}^{-1}\boldsymbol{L}=\boldsymbol{M}_{\Po}^{-1}-\frac{\boldsymbol{e}_{1,p}\boldsymbol{e}_{1,p}^T}{\boldsymbol{e}_{1,p}^T\boldsymbol{M}_{\Po}\boldsymbol{e}_{1,p}+\frac{b}{m}},\notag\\
\tr\bigl(\boldsymbol{M}\boldsymbol{M}_{\Po}^{-1}\bigr)&=\tr(\boldsymbol{L})=p-\frac{\boldsymbol{e}_{1,p}^T\boldsymbol{M}_{\Po}\boldsymbol{e}_{1,p}}{\boldsymbol{e}_{1,p}^T\boldsymbol{M}_{\Po}\boldsymbol{e}_{1,p}+\frac{b}{m}}.\notag
\end{align}
Let $\lambda_1=\exp\bigl(\boldsymbol{f}(\boldsymbol{d})^T\boldsymbol{\beta}-z^{\ast}\bigr)$ and $\lambda_p=\exp\bigl(\boldsymbol{f}(\boldsymbol{d})^T\boldsymbol{\beta}\bigr)$. For the design $\xi^{\ast}$ the information matrix for the Poisson model can be decomposed as $\boldsymbol{M}_{\Po}=\boldsymbol{X}^T\boldsymbol{\Lambda}\boldsymbol{W}\boldsymbol{X}$ with $\boldsymbol{\Lambda}=\Diag(\lambda_1,\ldots,\lambda_1,\lambda_p)$, $\boldsymbol{W}=\Diag(w_1^{\ast},\ldots,w_1^{\ast},w_p^{\ast})$ and
\begin{align}
\boldsymbol{X}=\begin{pmatrix}\boldsymbol{1}_{p-1} & & \boldsymbol{1}_{p-1}\cdot \boldsymbol{d}^T-z^{\ast}\cdot\Diag(\tilde{\boldsymbol{\beta}}_{\text{rec}})\\
1 & & \boldsymbol{d}^T\end{pmatrix}\notag
\end{align}
with $\boldsymbol{1}_{p-1}=(1,\ldots,1)^T\in\mathbb{R}^{p-1}$ and $\tilde{\boldsymbol{\beta}}_{\text{rec}}=(1/\beta_1,\ldots,1/\beta_{p-1})^T$, which has the reciprocal entries of $\tilde{\boldsymbol{\beta}}=(\beta_1,\ldots,\beta_{p-1})^T$. The inverse of $\boldsymbol{X}$ is given by
\begin{align}
\boldsymbol{X}^{-1}=\frac{1}{z^{\ast}}\cdot\begin{pmatrix}\bigl(\Diag(\tilde{\boldsymbol{\beta}})\boldsymbol{d}\bigr)^T & & z^{\ast}-\boldsymbol{d}^T\tilde{\boldsymbol{\beta}}\\
-\Diag(\tilde{\boldsymbol{\beta}}) & & \tilde{\boldsymbol{\beta}}\end{pmatrix}.\notag
\end{align}
Now, we show that equality holds in condition \eqref{Äquivalenzsatz Bedingung} at the support points of the design $\xi^{\ast}$. First we consider the support point $\boldsymbol{d}$. The inverse of the information matrix for the Poisson model is given by $\boldsymbol{M}_{\Po}^{-1}=\boldsymbol{X}^{-1}\boldsymbol{W}^{-1}\boldsymbol{\Lambda}^{-1}(\boldsymbol{X}^T)^{-1}$. With $(\boldsymbol{X}^T)^{-1}\boldsymbol{f}(\boldsymbol{d})=(\boldsymbol{X}^T)^{-1}\boldsymbol{X}^T\boldsymbol{e}_{p,p}=\boldsymbol{e}_{p,p}$ we have:
\begin{align}
g_i(d_i)&=\boldsymbol{f}\bigl(\boldsymbol{d}+(d_i-d_i)\cdot \boldsymbol{e}_i\bigr)^T\boldsymbol{M}_{\Po}^{-1}\boldsymbol{M}\boldsymbol{M}_{\Po}^{-1}\boldsymbol{f}\bigl(\boldsymbol{d}+(d_i-d_i)\cdot \boldsymbol{e}_i\bigr)\notag\\
&=\boldsymbol{f}(\boldsymbol{d})^T\cdot\left(\boldsymbol{X}^{-1}\boldsymbol{W}^{-1}\boldsymbol{\Lambda}^{-1}(\boldsymbol{X}^T)^{-1}-\frac{\boldsymbol{e}_{1,p}\boldsymbol{e}_{1,p}^T}{\boldsymbol{e}_{1,p}^T\boldsymbol{M}_{\Po}\boldsymbol{e}_{1,p}+\frac{b}{m}}\right)\cdot \boldsymbol{f}(\boldsymbol{d})\notag\\
&=\boldsymbol{e}_{p,p}^T\boldsymbol{W}^{-1}\boldsymbol{\Lambda}^{-1}\boldsymbol{e}_{p,p}-\frac{1}{\boldsymbol{e}_{1,p}^T\boldsymbol{M}_{\Po}\boldsymbol{e}_{1,p}+\frac{b}{m}}=\frac{1}{w_p^{\ast}\cdot \lambda_p}-\frac{1}{\boldsymbol{e}_{1,p}^T\boldsymbol{M}_{\Po}\boldsymbol{e}_{1,p}+\frac{b}{m}}.\notag
\end{align}
It follows:
\begin{align}
h_i(d_i)&=g_i(d_i)-\exp\bigl(-\boldsymbol{f}(\boldsymbol{d}+(d_i-d_i)\cdot \boldsymbol{e}_i)^T\boldsymbol{\beta}\bigr)\cdot\tr\bigl(\boldsymbol{M}\boldsymbol{M}_{\Po}^{-1}\bigr)\notag\\
&=\frac{\frac{1}{w_p^{\ast}\cdot \lambda_p}\cdot\left(\boldsymbol{e}_{1,p}^T\boldsymbol{M}_{\Po}\boldsymbol{e}_{1,p}+\frac{b}{m}\right)-1}{\boldsymbol{e}_{1,p}^T\boldsymbol{M}_{\Po}\boldsymbol{e}_{1,p}+\frac{b}{m}}-\frac{1}{\lambda_p}\cdot\frac{(p-1)\cdot \boldsymbol{e}_{1,p}^T\boldsymbol{M}_{\Po}\boldsymbol{e}_{1,p}+p\cdot\frac{b}{m}}{\boldsymbol{e}_{1,p}^T\boldsymbol{M}_{\Po}\boldsymbol{e}_{1,p}+\frac{b}{m}}\notag\\
&=\frac{\boldsymbol{e}_{1,p}^T\boldsymbol{M}_{\Po}\boldsymbol{e}_{1,p}-w_p^{\ast}\cdot \lambda_p-(p-1)\cdot w_p^{\ast}\cdot \boldsymbol{e}_{1,p}^T\boldsymbol{M}_{\Po}\boldsymbol{e}_{1,p}+\frac{b}{m}\cdot(1-p\cdot w_p^{\ast})}{w_p^{\ast}\cdot \lambda_p\cdot\left(\boldsymbol{e}_{1,p}^T\boldsymbol{M}_{\Po}\boldsymbol{e}_{1,p}+\frac{b}{m}\right)}.\notag
\end{align}
With $\boldsymbol{e}_{1,p}^T\boldsymbol{M}_{\Po}\boldsymbol{e}_{1,p}=(p-1)\cdot w_1^{\ast}\cdot \lambda_1+w_p^{\ast}\cdot \lambda_p$ and $(1-p\cdot w_p^{\ast})/(p-1)=-1+p\cdot w_1^{\ast}$ we obtain:
\begin{align}
h_i(d_i)&=\frac{(p-1)\cdot w_1^{\ast}\cdot \lambda_1-(p-1)\cdot w_p^{\ast}\cdot \boldsymbol{e}_{1,p}^T\boldsymbol{M}_{\Po}\boldsymbol{e}_{1,p}+\frac{b}{m}\cdot(1-p\cdot w_p^{\ast})}{w_p^{\ast}\cdot \lambda_p\cdot\left(\boldsymbol{e}_{1,p}^T\boldsymbol{M}_{\Po}\boldsymbol{e}_{1,p}+\frac{b}{m}\right)}\notag\\
&=(p-1)\cdot\frac{w_1^{\ast}\cdot \lambda_1-w_p^{\ast}\cdot \boldsymbol{e}_{1,p}^T\boldsymbol{M}_{\Po}\boldsymbol{e}_{1,p}+\frac{b}{m}\cdot(-1+p\cdot w_1^{\ast})}{w_p^{\ast}\cdot \lambda_p\cdot\left(\boldsymbol{e}_{1,p}^T\boldsymbol{M}_{\Po}\boldsymbol{e}_{1,p}+\frac{b}{m}\right)}.\notag
\end{align}
Since the numerator is equal to zero by Lemma \ref{Lemma D-optimale Gewichte Gleichung Poisson-Gamma-Modell}, it follows that $h_i(d_i)=0$.\\
Now, it is shown that equality in condition \eqref{Äquivalenzsatz Bedingung} also holds at the support points $\boldsymbol{d}-(z^{\ast}/\beta_i)\cdot \boldsymbol{e}_i$, $i=1,\ldots,p-1$. Using the relation $(\boldsymbol{X}^T)^{-1}\boldsymbol{f}\bigl(\boldsymbol{d}-(z^{\ast}/\beta_i)\cdot \boldsymbol{e}_i\bigr)=(\boldsymbol{X}^T)^{-1}\boldsymbol{X}^T\boldsymbol{e}_{i,p}=\boldsymbol{e}_{i,p}$ we obtain:
\begin{align}
&g_i\negthinspace\left(d_i-\frac{z^{\ast}}{\beta_i}\right)\notag\\
&=\boldsymbol{f}\negthinspace\left(\boldsymbol{d}-\frac{z^{\ast}}{\beta_i}\cdot \boldsymbol{e}_i\right)^T\cdot\left(\boldsymbol{X}^{-1}\boldsymbol{W}^{-1}\boldsymbol{\Lambda}^{-1}(\boldsymbol{X}^T)^{-1}-\frac{\boldsymbol{e}_{1,p}\boldsymbol{e}_{1,p}^T}{\boldsymbol{e}_{1,p}^T\boldsymbol{M}_{\Po}\boldsymbol{e}_{1,p}+\frac{b}{m}}\right)\cdot \boldsymbol{f}\negthinspace\left(\boldsymbol{d}-\frac{z^{\ast}}{\beta_i}\cdot \boldsymbol{e}_i\right)\notag\\
&=\boldsymbol{e}_{i,p}^T\boldsymbol{W}^{-1}\boldsymbol{\Lambda}^{-1}\boldsymbol{e}_{i,p}-\frac{1}{\boldsymbol{e}_{1,p}^T\boldsymbol{M}_{\Po}\boldsymbol{e}_{1,p}+\frac{b}{m}}=\frac{1}{w_1^{\ast}\cdot \lambda_1}-\frac{1}{\boldsymbol{e}_{1,p}^T\boldsymbol{M}_{\Po}\boldsymbol{e}_{1,p}+\frac{b}{m}}.\notag
\end{align}
With $\exp\bigl(\boldsymbol{f}(\boldsymbol{d}-(z^{\ast}/\beta_i)\cdot \boldsymbol{e}_i)^T\boldsymbol{\beta}\bigr)=\exp\left(\boldsymbol{f}(\boldsymbol{d})^T\boldsymbol{\beta}-z^{\ast}\right)=\lambda_1$ we have:
\begin{align}
h_i\negthinspace\left(d_i-\frac{z^{\ast}}{\beta_i}\right)&=\frac{\frac{1}{w_1^{\ast}\cdot \lambda_1}\cdot\left(\boldsymbol{e}_{1,p}^T\boldsymbol{M}_{\Po}\boldsymbol{e}_{1,p}+\frac{b}{m}\right)-1}{\boldsymbol{e}_{1,p}^T\boldsymbol{M}_{\Po}\boldsymbol{e}_{1,p}+\frac{b}{m}}-\frac{1}{\lambda_1}\cdot\frac{(p-1)\cdot \boldsymbol{e}_{1,p}^T\boldsymbol{M}_{\Po}\boldsymbol{e}_{1,p}+p\cdot\frac{b}{m}}{\boldsymbol{e}_{1,p}^T\boldsymbol{M}_{\Po}\boldsymbol{e}_{1,p}+\frac{b}{m}}\notag\\
&=\frac{\boldsymbol{e}_{1,p}^T\boldsymbol{M}_{\Po}\boldsymbol{e}_{1,p}\cdot\left(1-(p-1)\cdot w_1^{\ast}\right)-w_1^{\ast}\cdot \lambda_1+\frac{b}{m}\cdot(1-p\cdot w_1^{\ast})}{w_1^{\ast}\cdot \lambda_1\cdot\left(\boldsymbol{e}_{1,p}^T\boldsymbol{M}_{\Po}\boldsymbol{e}_{1,p}+\frac{b}{m}\right)}.\notag
\end{align}
Since $1-(p-1)\cdot w_1^{\ast}=w_p^{\ast}$, it follows that $h_i(d_i-z^{\ast}/\beta_i)=0$ holds for $i=1,\ldots,p-1$ by Lemma \ref{Lemma D-optimale Gewichte Gleichung Poisson-Gamma-Modell}.\\
The first and second derivative of $h_i$ are given by:
\begin{align}
h_i'(x)&=g_i'(x)+\beta_i\cdot \exp\bigl(-\boldsymbol{f}(\boldsymbol{d}+(x-d_i)\cdot \boldsymbol{e}_i)^T\boldsymbol{\beta}\bigr)\cdot\tr\bigl(\boldsymbol{M}\boldsymbol{M}_{\Po}^{-1}\bigr),\notag\\
h_i''(x)&=g_i''(x)-\beta_i^2\cdot \exp\bigl(-\boldsymbol{f}(\boldsymbol{d}+(x-d_i)\cdot \boldsymbol{e}_i)^T\boldsymbol{\beta}\bigr)\cdot\tr\bigl(\boldsymbol{M}\boldsymbol{M}_{\Po}^{-1}\bigr).\notag
\end{align}
Since $g_i$ is a quadratic polynomial, its second derivative $g_i''$ is constant. The exponential function is injective, so $h_i''$ is also injective and has at most one zero. By Rolle's theorem $h_i'$ can have no more than two zeros. Therefore, $h_i$ has at most two extrema. Moreover, we have $\lim_{x\rightarrow\pm\infty} h_i(x)=\pm\infty$ for $\beta_i>0$ and $\lim_{x\rightarrow\pm\infty} h_i(x)=\mp\infty$ for $\beta_i<0$. Hence there are no saddle points. The design $\xi^{\ast}$ is $D$-optimal if $h_i$ has a maximum at $d_i-z^{\ast}/\beta_i$ for $i=1,\ldots,p-1$. Because of $h_i(d_i)=h_i(d_i-z^{\ast}/\beta_i)=0$ there is a minimum between the two support points. A maximum occurs at $d_i-z^{\ast}/\beta_i$ if $h_i'$ has a zero at $d_i-z^{\ast}/\beta_i$. The derivative of $g_i$ is given by:
\begin{align}
g_i'(x)&=2\cdot \boldsymbol{e}_{i+1,p}^T\boldsymbol{M}_{\Po}^{-1}\boldsymbol{M}\boldsymbol{M}_{\Po}^{-1}\boldsymbol{f}\bigl(\boldsymbol{d}+(x-d_i)\cdot \boldsymbol{e}_i\bigr)\notag\\
&=2\cdot \boldsymbol{e}_{i+1,p}^T\cdot\left(\boldsymbol{X}^{-1}\boldsymbol{W}^{-1}\boldsymbol{\Lambda}^{-1}(\boldsymbol{X}^T)^{-1}-\frac{\boldsymbol{e}_{1,p}\boldsymbol{e}_{1,p}^T}{\boldsymbol{e}_{1,p}^T\boldsymbol{M}_{\Po}\boldsymbol{e}_{1,p}+\frac{b}{m}}\right)\cdot \boldsymbol{f}\bigl(\boldsymbol{d}+(x-d_i)\cdot \boldsymbol{e}_i\bigr).\notag
\end{align}
With $\boldsymbol{e}_{i+1,p}^T\boldsymbol{X}^{-1}=(\beta_i/z^{\ast})\cdot(-\boldsymbol{e}_i^T,1)$ and $(\boldsymbol{X}^T)^{-1}\boldsymbol{f}\bigl(\boldsymbol{d}-(z^{\ast}/\beta_i)\cdot \boldsymbol{e}_i\bigr)=\boldsymbol{e}_{i,p}$ it follows:
\begin{align}
p'\negthinspace\left(d_i-\frac{z^{\ast}}{\beta_i}\right)&=2\cdot \boldsymbol{e}_{i+1,p}^T\cdot\negthinspace\left(\boldsymbol{X}^{-1}\boldsymbol{W}^{-1}\boldsymbol{\Lambda}^{-1}(\boldsymbol{X}^T)^{-1}-\frac{\boldsymbol{e}_{1,p}\boldsymbol{e}_{1,p}^T}{\boldsymbol{e}_{1,p}^T\boldsymbol{M}_{\Po}\boldsymbol{e}_{1,p}+\frac{b}{m}}\right)\negthinspace\cdot \boldsymbol{f}\negthinspace\left(\boldsymbol{d}-\frac{z^{\ast}}{\beta_i}\cdot \boldsymbol{e}_i\right)\notag\\
&=2\cdot\frac{\beta_i}{z^{\ast}}\cdot(-\boldsymbol{e}_i^T,1)\boldsymbol{W}^{-1}\boldsymbol{\Lambda}^{-1}\boldsymbol{e}_{i,p}-2\cdot\frac{\boldsymbol{e}_{i+1,p}^T\boldsymbol{e}_{1,p}\boldsymbol{e}_{1,p}^T\boldsymbol{f}\bigl(\boldsymbol{d}-\frac{z^{\ast}}{\beta_i}\cdot \boldsymbol{e}_i\bigr)}{\boldsymbol{e}_{1,p}^T\boldsymbol{M}_{\Po}\boldsymbol{e}_{1,p}+\frac{b}{m}}\notag\\
&=-2\cdot\frac{\beta_i}{z^{\ast}}\cdot\frac{1}{w_1^{\ast}\cdot \lambda_1}.\notag
\end{align}
Here it was used that $\boldsymbol{e}_{i+1,p}^T\boldsymbol{e}_{1,p}=0$. We have:
\begin{align}
h_i'\negthinspace\left(d_i-\frac{z^{\ast}}{\beta_i}\right)&=-2\cdot\frac{\beta_i}{z^{\ast}}\cdot\frac{1}{w_1^{\ast}\cdot \lambda_1}+\frac{\beta_i}{\lambda_1}\cdot\frac{(p-1)\cdot \boldsymbol{e}_{1,p}^T\boldsymbol{M}_{\Po}\boldsymbol{e}_{1,p}+p\cdot\frac{b}{m}}{\boldsymbol{e}_{1,p}^T\boldsymbol{M}_{\Po}\boldsymbol{e}_{1,p}+\frac{b}{m}}\notag\\
&=\frac{\beta_i}{\lambda_1}\cdot\frac{-\frac{2}{z^{\ast}\cdot w_1^{\ast}}\cdot\left(\boldsymbol{e}_{1,p}^T\boldsymbol{M}_{\Po}\boldsymbol{e}_{1,p}+\frac{b}{m}\right)+(p-1)\cdot \boldsymbol{e}_{1,p}^T\boldsymbol{M}_{\Po}\boldsymbol{e}_{1,p}+p\cdot\frac{b}{m}}{\boldsymbol{e}_{1,p}^T\boldsymbol{M}_{\Po}\boldsymbol{e}_{1,p}+\frac{b}{m}}.\notag
\end{align}
The equation $h_i'(d_i-z^{\ast}/\beta_i)=0$ is equivalent to:
\begin{align}
0&=-2\cdot\left(m\cdot\boldsymbol{e}_{1,p}^T\boldsymbol{M}_{\Po}\boldsymbol{e}_{1,p}+b\right)+z^{\ast}\cdot w_1^{\ast}\cdot\left(m\cdot(p-1)\cdot \boldsymbol{e}_{1,p}^T\boldsymbol{M}_{\Po}\boldsymbol{e}_{1,p}+p\cdot b\right)\notag\\
&=m\cdot\boldsymbol{e}_{1,p}^T\boldsymbol{M}_{\Po}\boldsymbol{e}_{1,p}\cdot\big((p-1)\cdot z^{\ast}\cdot w_1^{\ast}-2\big)+b\cdot\left(p\cdot z^{\ast}\cdot w_1^{\ast}-2\right)\notag\\
&=m\cdot\big((p-1)\cdot w_1^{\ast}\cdot \lambda_1+w_p^{\ast}\cdot \lambda_p\big)\cdot\big((p-1)\cdot z^{\ast}\cdot w_1^{\ast}-2\big)+b\cdot\left(p\cdot z^{\ast}\cdot w_1^{\ast}-2\right).\notag
\end{align}
This equation has a solution, which is shown by using the intermediate value theorem.
Since $\lim_{z^{\ast}\rightarrow0} \lambda_1=\lambda_p$, we have:
\begin{align}
&\lim_{z^{\ast}\rightarrow0} m\cdot\big((p-1)\cdot w_1^{\ast}\cdot \lambda_1+w_p^{\ast}\cdot \lambda_p\big)\cdot\big((p-1)\cdot z^{\ast}\cdot w_1^{\ast}-2\big)+b\cdot\left(p\cdot z^{\ast}\cdot w_1^{\ast}-2\right)\notag\\
&=-2\cdot m\cdot \lambda_p-2\cdot b<0.\notag
\end{align}
Moreover, since $1/p<w_1^{\ast}<1/(p-1)$ by Theorem \ref{Satz D-optimale Gewichte Poisson-Gamma-Modell}, it follows:
\begin{align}
\lim_{z^{\ast}\rightarrow\infty} m\cdot\big((p-1)\cdot w_1^{\ast}\cdot \lambda_1+w_p^{\ast}\cdot \lambda_p\big)\cdot\big((p-1)\cdot z^{\ast}\cdot w_1^{\ast}-2\big)+b\cdot\left(p\cdot z^{\ast}\cdot w_1^{\ast}-2\right)=\infty.\notag
\end{align}
Hence, the equation has a solution in the interval $(0,\infty)$. Thus the design $\xi^{\ast}$ is $D$-optimal. If the equation had more than one solution in the interval $(0,\infty)$, then another $D$-optimal design exists. Since the criterion function is concave, any convex combination of these $D$-optimal designs would also be $D$-optimal. The resulting design would have more than two support points on an edge, which is a contradiction to the structure of $h_i$. Thus, the solution of the equation in the interval $(0,\infty)$ is unique and so is the $D$-optimal design. If $z^{\ast}\leq\min_{i=1,\ldots,p-1}\bigl(\left|\beta_i\right|\cdot(v_i-u_i)\bigr)$, then all support points of $\xi^{\ast}$ are located within $\mathscr{X}\subseteq\mathscr{X}_{\text{ext}}$. Hence, the design $\xi^{\ast}$ is $D$-optimal on $\mathscr{X}$.
\end{proof}
\begin{proof}[Proof of Theorem \ref{Satz DA-optimale Design}]\ \\
With Lemma \ref{Lemma Verallgemeinerte Inverse Informationsmatrix Poisson-Gamma-Modell} for the generalized inverse of the information matrix the following equivalence for the criterion functions for the Poisson-Gamma and Poisson model holds:
\begin{align}
\det\bigl(\boldsymbol{A}^T \boldsymbol{M}(\xi;\boldsymbol{\beta})^{-} \boldsymbol{A}\bigr)&=\det\left(\boldsymbol{A}^T\cdot\left(\frac{b}{a}\cdot \boldsymbol{M}_{\Po}(\xi;\boldsymbol{\beta})^-+\frac{m}{a}\cdot \boldsymbol{e}_1\boldsymbol{e}_1^T\right)\boldsymbol{A}\right)\notag\\
&=\left(\frac{b}{a}\right)^s\cdot\det\bigl(\boldsymbol{A}^T \boldsymbol{M}_{\Po}(\xi;\boldsymbol{\beta})^{-} \boldsymbol{A}\bigr).\notag
\end{align}
It was used that $\boldsymbol{A}^T\boldsymbol{e}_1=\boldsymbol{0}_s$. Since the identifiability condition is equivalent in both models by Theorem \ref{Satz Identifizierbarkeit}, the $D_A$-optimal designs coincide.
\end{proof}


\begin{thebibliography}{99}
\setlength{\itemsep}{0pt}
\bibitem{Atkinson} Atkinson, A. C., Donev, A. N. and Tobias, R. D. (2007). \textit{Optimum Experimental Designs, with SAS}. Oxford University Press, Oxford.
\bibitem{Chernoff} Chernoff, H. (1953). \textit{Locally optimal designs for estimating parameters}. The Annals of Mathematical Statistics 24, 586-602.
\bibitem{Fedorov1} Fedorov, V. V. (1972). \textit{Theory of optimal experiments}. Academic Press, New York.
\bibitem{Fedorov2} Fedorov, V. V. and Hackl, P. (1997). \textit{Model-Oriented Design of Experiments}. Springer, New York.
\bibitem{Ford} Ford, I., Torsney, B. und Wu, C. F. J. (1992). \textit{The Use of a Canonical Form in the Construction of Locally Optimal Designs for Non-Linear Problems}. Journal of the Royal Statistical Society, Series B 54, 569-583. 
\bibitem{Grasshoff1} Graßhoff, U., Holling, H. and Schwabe, R. (2016). \textit{Optimal Design for the Rasch Poisson-Gamma Model}. In mODa\,11 - Advances in Model-Oriented Design and Analysis (eds.: Kunert, J., Müller C. H. and Atkinson, A. C.), Springer, 133--141.
\bibitem{Grasshoff2} Graßhoff, U., Holling, H. and Schwabe, R. (2018). \textit{D-optimal Design for the Poisson Regression Models}. Technical Report, Magdeburg.
\bibitem{Mueller} Müller, C. H. (1995). \textit{Maximin efficient designs for estimating nonlinear aspects in linear models}. J. Statist. Plann. Inference, 44, 117--132.
\bibitem{Niaparast} Niaparast, M. (2009). \textit{Optimal Designs For Mixed Effects Poisson Regression Models}. PhD thesis, Faculty of Mathematics, Magdeburg.
\bibitem{Pronzato} Pronzato, L. and Pázman, A. (2013). \textit{Design of Experiments in Nonlinear Models}. Springer, New York.
\bibitem{Rodriguez} Rodríguez-Torreblanca, C. and Rodríguez-Díaz, J. M. (2007). \textit{Locally D- and c-optimal designs for Poisson and negative binomial regression models}. Metrika, 66, 161--172.
\bibitem{Russell} Russell, K. G., Woods, D. C., Lewis, S. M. and Eccleston, E. C. (2009). \textit{D-optimal designs for Poisson regression models}. Statistica Sinica, 19, 721--730.
\bibitem{Schmelter} Schmelter, T. (2007). \textit{The optimality of single-group designs for certain mixed models}. Metrika, 65, 183--193.
\bibitem{Schmidt1} Schmidt, D. and Schwabe, R. (2017). \textit{Optimal Design for Multiple Regression with Information Driven by the Linear Predictor}. Statistica Sinica, 27, 1371--1384.
\bibitem{Schmidt2} Schmidt, D. (2018). \textit{Characterization of $c$-, $L$- and $\phi_k$-optimal designs for a class of Nonlinear Multiple Regression Models}. To appear: Journal of the Royal Statistical Society, Series B.
\bibitem{Silvey} Silvey, S. D. (1980). \textit{Optimal Design}. Chapman and Hall, London.
\bibitem{Wang} Wang, Y., Myers, R. H., Smith, E. P. and Ye, K. (2006). \textit{D-optimal designs for Poisson regression models}. J. Statist. Plann. Inference, 136, 2831--2845.
\end{thebibliography}
\end{document}